\documentclass[11pt,twoside,a4paper]{article}
\usepackage[english]{babel}
\usepackage{a4wide}
\usepackage{graphicx}
\usepackage{amsmath,amssymb,amsthm, amsgen}
\usepackage{listings}
\usepackage{color}
\usepackage[usenames,dvipsnames]{xcolor}
\usepackage{rotating}
\usepackage{hhline}
\usepackage{multirow}
\usepackage{enumerate}
\usepackage[bookmarks]{}
\usepackage{amsfonts}   
\usepackage{dsfont}
\usepackage{tikz}
\usepackage{booktabs}
\usepackage{longtable}
\usepackage{enumitem}
\usepackage{stmaryrd} 
\usepackage{epstopdf}
\usepackage{comment}
\usepackage{epstopdf}

\usepackage{thmtools}
\newtheorem{thm}{Theorem}[section]
\newtheorem{prop}[thm]{Proposition}
\newtheorem{lem}[thm]{Lemma}
\newtheorem{lemma}[thm]{Lemma}

\newtheorem{defn}[thm]{Definition}

\newtheorem{rem}[thm]{Remark}
\newtheorem{ass}[thm]{Assumption}

\numberwithin{equation}{section}

\def\XXint#1#2#3{{\setbox0=\hbox{$#1{#2#3}{\int}$ }
\vcenter{\hbox{$#2#3$ }}\kern-.6\wd0}}

\newcommand{\e}{\varepsilon}

\newcommand{\N}{\mathbb N}
\newcommand{\R}{\mathbb R}

\newcommand{\Z}{\mathbb Z}

\newcommand{\bb}{\mathbf b}

\newcommand{\bx}{\mathbf x}

\newcommand{\cI}{\mathcal I}

\newcommand{\cT}{\mathcal T}

\newcommand{\tT}{\tilde T}
\newcommand{\ttau}{\tilde \tau}
\newcommand{\tv}{\tilde v}
\newcommand{\tx}{\tilde x}

\newcommand{\osT}{\overline{\mathsf T}}
\newcommand{\hsT}{\hat{\mathsf T}}

\newcommand{\beq}{\begin{equation}}
\newcommand{\eeq}{\end{equation}}
\newcommand{\beqs}{\begin{equation*}}
\newcommand{\eeqs}{\end{equation*}}
\newcommand{\beqa}{\begin{eqnarray}}
\newcommand{\eeqa}{\end{eqnarray}}
\newcommand{\beqas}{\begin{eqnarray*}}
\newcommand{\eeqas}{\end{eqnarray*}}
\newcommand{\xs}{\overline{x}}
\newcommand{\vs}{\overline{v}}
\newcommand{\cs}{\overline{c}}
\newcommand{\ep}{\varepsilon}
\newcommand{\xss}{\hat{x}}

\renewcommand{\o}{\overline}

\title{Discrete Dislocation Dynamics with annihilation as the limit of the Peierls-Nabarro model in one dimension}

\author{Patrick van Meurs, Stefania Patrizi}

\def\showfigbool{1}
\long\def\showfig#1{{\if\showfigbool1 {#1} \fi}}

\begin{document}

\maketitle






\begin{abstract}
Plasticity of metals is the emergent phenomenon of many crystal defects (dislocations) which interact and move on microscopic time and length scales. Two of the commonly used models to describe such dislocation dynamics are the Peierls-Nabarro model and the so-called discrete dislocation dynamics model. 
 However, the consistency between these two models is known only for a few number of dislocations or up to the first time at which two dislocations collide. In this paper we resolve these restrictions, and establish the consistency for any number of dislocations and without any restriction on their initial position or orientation.

In more detail, the evolutive Peierls-Nabarro model which we consider describes the evolution of a phase-field function $v_\e(t,x)$ which represents the atom deformation in a crystal. The model is a reaction-diffusion  equation of Allen-Cahn type with the half Laplacian. The small parameter $\ep$ is the ratio between the atomic distance and the typical distance between phase transitions in $v_\e$. The position of a phase transition determines the position of a dislocation, and the sign of the transition (up or down) determines the orientation. 

The goal of this paper is to derive the asymptotic behavior of the function $v_\e$ as $\ep\to0$ up to arbitrary end time $T$; in particular beyond collisions. We prove that $v_\e$ converges to a piecewise constant function $v$, whose jump points in the spatial variable satisfy the ODE system which represents discrete dislocation dynamics with annihilation. Our proof method is to explicitly construct and patch together several sub- and supersolutions of $v_\e$, and to show that they converge to the same limit $v$.
\end{abstract}

\tableofcontents

\bigskip

\section{Introduction}


In three dimensions dislocations are line defects in crystals. These
lines can move in the crystallographic planes (slip planes), which typically happens when the crystal is submitted to shear stress. This movement
is one of the main explanations for the plastic behavior of metals. We refer the reader to the books \cite{hl,HullBacon11} for a tour in the theory of dislocations.
Dislocations  can be described at several scales by different models, see e.g.  the review paper \cite{DipierroPatriziValdinoci22}.
Due to the complexity of the evolution of lines in three dimensions, we follow the common simplification in which it is assumed that all dislocations are straight and parallel edge dislocations which lie and move in the same slip plane. Then, the dislocations can be represented by points on a line which lies in the slip plane and is perpendicular to the dislocation lines. In addition to their position, each dislocation has either a  positive or a  negative orientation.  

Even after this simplification to reduce the number of spatial dimensions from three to one, there are several different models to describe the dynamics of dislocations. Here, we mention three of these models. First, the classical model by Frenkel and Kontorova is the most detailed model out of the three; it describes the deformation of all the atoms. The dislocation positions appear in the form of certain local configurations of the atoms. Second, discrete dislocation dynamics is the coarsest model out of the three; it neglects any atomic effects and simply describes the dislocation positions as a system of ODEs on the continuous line $\R$. Third, the Peierls-Nabarro model lies in between these two; it does not describe the position of each atom, but it does describe the atom positions through a continuous displacement function $v_\e$. In \cite{fino} the connection between the Frenkel-Kontorova model and the Peierls-Nabarro model is established, but a complete connection between the Peierls-Nabarro model and discrete dislocation dynamics appears to be missing in the literature. In this paper we fill this gap.

\subsection{Position in the literature}
We are not the first to attempt to fill the aforementioned gap in the literature. Therefore, we give a brief overview of the literature on the connection between the Peierls-Nabarro model and discrete dislocation dynamics. With this aim, we first describe the Peierls-Nabarro model in more detail. This model is based on the classical work by
Peierls and Nabarro \cite{Peierls40,Nabarro47}. It describes the evolution of 
the displacement function $v_\ep(t,x)$ as a phase field. The function $v_\ep(t,x)$ depends
on the time variable~$t\geq0$,  the space variable~$x\in\R$, and a small parameter $\e > 0$ which is the ratio between the atomic distance and the typical distance between two neighboring dislocations. Figure \ref{fig:ve} illustrates a typical profile of $v_\ep(t,\cdot)$. The position of a transition layer determines the position of a dislocation, and the sign of the transition (up or down) determines the orientation. 
We allow for an arbitrary number of adjacent dislocations with the same orientation, and therefore $v_\e$ may attain an arbitrary number of different phases. 

\begin{figure}[ht]
\centering
\begin{tikzpicture}[scale=1.5, >= latex]
\def \rr {0.04} 

\draw[->] (.5,-.2) -- (.5, 2.2);
\draw[->] (0,0) -- (5, 0) node[below]{$x$};
\draw[thin] (0,1) -- (5, 1);
\draw[thin] (0,2) -- (5, 2);

\draw [thick, red] (0,0) .. controls (.9,0) .. (1,.5);
\draw [thick, red] (1,.5) .. controls (1.1,1) .. (1.5,1);
\draw [thick, red] (1.5,1) .. controls (1.9,1) .. (2,1.5);
\draw [thick, red] (2,1.5) .. controls (2.1,2) .. (3,2);
\draw [thick, red] (3,2) .. controls (3.9,2) .. (4,1.5) node[right]{$v_\e(t,x)$};
\draw [thick, red] (4,1.5) .. controls (4.1,1) .. (5,1);

\draw (.8,0) rectangle (1.2,1);
\draw[<->] (.2,0) --++ (0,1) node[midway, left]{$1$};
\draw[<->] (.8,1.1) --++ (.4,0) node[midway, above]{$\e$};
\draw[thin] (1,.5) --++ (0,-.5) node[below]{$x_1$};
\draw[thin] (2,1.5) --++ (0,-1.5) node[below]{$x_2$};
\draw[thin] (4,1.5) --++ (0,-1.5) node[below]{$x_3$};
\end{tikzpicture} \\
\caption{Sketch of $v_\e(t,x)$ at a given time $t$.}
\label{fig:ve}
\end{figure}
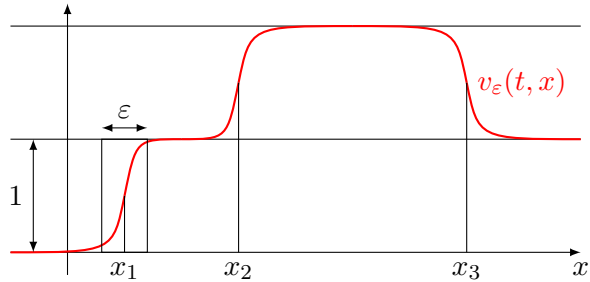

The Peierls-Nabarro model is a nonlocal equation for $v_\e$.
More precisely, the influence of the elastic energy of the whole
crystal along the slip plane produces a fractional operator (the half Laplacian),
which we denote by~$\cI$. The contribution of $\cI$ is balanced by an atomic force which pushes the atoms to the lattice positions $\e \Z$, and can therefore be written as the negative of the derivative of a periodic multi-well
potential~$W$.
The long time behavior of $v_\ep$ is studied in \cite{PatriziValdinoci17}.

In view of our aim, we are interested in the limit $\e \to 0$. Several results have already been obtained. Starting from an initial configuration where the dislocation transitions occurs at some given points, the displacement  function~$v_\ep$
approaches as $\e \to 0$ a piecewise constant function $v$
(see~\cite{GonzalezMonneau12, dfv14,dpv15,PatriziValdinoci15, PatriziValdinoci16}).
The plateaus of this asymptotic limit correspond to
the periodic sites induced by the crystalline structure. The jump points~$x_1(t),\dots,x_N(t)$ evolve in time 
as described by the discrete dislocation dynamics model. This model is a system of first order ordinary differential equations for `the particles' $x_i(t)$ which is driven by a singular interaction potential. 
We refer to Subsection 2.2 of  \cite{PatriziValdinoci15} for a  heuristic discussion  of the link between the integro-differential equation of $v_\e$
and the ODE system for $x_i(t)$.   
Remarkably, the physical properties of the singular potential
of this ODE system depend
on the orientation of the dislocation
at the jump points. Namely, if the jumps of $v$ at~$x_i$ and~$x_{i+1}$
are in the same direction (either up or down),
then the potential induces a repulsion between
the particles~$x_i$ and~$x_{i+1}$.
Conversely, when the jumps of $v$ are in opposite direction,
then the potential becomes attractive, and the two particles may
collide in finite time. However, the singularity of the potential complicates the analysis of the ODE system of $x_i(t)$ close to and at collisions. In addition, it is challenging to get sufficient estimates on the solution $v_\e$ of the Peierls-Nabarro model when two phase transitions with opposite jumps approach each other. This is the main reason that the current results on the limit passage as $\e \to 0$ of $v_\e$ are limited to either the first time at which two dislocations collide, or to a few number of dislocations (2 or 3). In more detail for the second case, in \cite{PatriziValdinoci16} the second author and Valdinoci  studied the behavior of the displacement  function $v_\ep$ across the collision time in  
the case of two and three
particles and showed that the limit configuration of~$v_\ep$ after collision 
is either a constant (in the case of two particles) or a constant-in-time simple  jump function 
(in the case of three particles). 

Recently, the first author, Peletier and Pozar   \cite{VanMeursPeletierPozarXX} proved well-posedness for the  ODE system for $x_i(t)$ which allows to resolve collisions and continue the evolution beyond collision times. More details will be given below in Section \ref{s:intro:main}. This result, together with \cite{PatriziValdinoci16}, provides us with sufficient tools to characterize the limit $\ep\to0$ of $v_\ep(t,x)$ without any limitations on $t$, on the number of dislocations, on their positions or on their orientations. This limit passage is the main result of our paper; see Theorem \ref{t} below.

\subsection{Main result}
\label{s:intro:main}
 
To describe Theorem \ref{t}, our main result, we introduce the Peierls-Nabarro model and the discrete dislocation dynamics model in full detail. We start with the Peierls-Nabarro model. It is given by
\begin{equation} \label{HJe:formal} \tag{HJ$_\e$}
\begin{cases}
  \e \partial_t v_\e = \cI [v_\e]  - \frac1\e W' ( v_\e )
  & \text{in } (0,\infty) \times \R\\
  v_\e(0,x)=v_\e^0(x)& \text{in }   \R,
  \end{cases}
\end{equation}
where $\cI$ is the half Laplacian $-(-\Delta)^\frac12$ defined by 
\begin{equation} \label{cI}
  \cI [\varphi] (x)
      := \text{PV}\int_\R \big( \varphi(x + z) - \varphi(x) \big) \, \frac{dz}{z^2}
      \qquad \text{for all } \varphi \in C_b^2(\R),
\end{equation}
where PV stands for principal value.   We refer to \cite{Silvestre} for a basic introduction to the fractional Laplace operator. Furthermore, $W$ in \eqref{HJe:formal} is a multi-well potential, which we assume throughout this paper to satisfy
\begin{equation}\label{Wass}
\begin{cases}W\in C^{2,\beta}(\R)& \text{for some }0<\beta<1\\
W(v+1)=W(v)& \text{for any } v\in\R\\
W(v)>0& \text{for any }0 < v < 1\\
W(0)=0 & \\
W''(0)>0. & \\
\end{cases}
\end{equation}
 A prototypical example of $W$ is $W(x) = \sin^2 (\pi x)$. Finally, $v_\ep^0$ is a given initial condition. We assume that it is a superposition of $N$  transition layers, each with an arbitrary orientation. More precisely, we take the positions of the transition layers from 
\begin{equation} \label{Om}
  \Omega^N := \{ \bx \in \R^N : x_1 < x_2 < \ldots < x_N \},
\end{equation}
and set $\bb \in \{-1, +1\}^N$ as a given list of orientations. To describe the transition layers at $x_i$, we first define the upward phase-transition profile
\begin{equation*}
  u(x) := u(x ; + 1) 
\end{equation*}
centred at $0$ on the atomic scale as the  so-called basic layer solution $u$ associated to $\cI$, that is the solution of
\begin{equation}\label{u}
\begin{cases}\cI(u) = W'(u) &\text{in}\quad \R\\
 u' > 0 &\text{in}\quad \R\\
\displaystyle\lim_{x\rightarrow-\infty} u(x) = 0,\quad \displaystyle\lim_{x\rightarrow+\infty} u(x) = 1,\quad u (0) = \displaystyle \frac{1}{2}.
\end{cases}
\end{equation}
From the properties of $\cI$ and $W$ it is then easy to check that the downward phase-transition profile
\begin{equation*} 
  u(x; -1) := u(-x) - 1
\end{equation*}
is the  solution of  
\begin{equation*} 
\begin{cases}\cI(v) = W'(v) &\text{in}\quad \R\\
 v' < 0 &\text{in}\quad \R\\
\displaystyle\lim_{x\rightarrow-\infty} v(x) = 0,\quad \displaystyle\lim_{x\rightarrow+\infty} v(x) = -1,\quad  v(0) = - \displaystyle \frac{1}{2}.
\end{cases}
\end{equation*}
Putting this together, our assumption on the initial data is as follows:
 
\begin{ass}[Well-prepared initial data] \label{a:v0e}
The initial condition $v_\e^0 : \R \to \R$ is well-prepared, i.e.\ there exist $\bx^0, \bx_\e^0 \in \Omega^N$ and a perturbation $\phi_\e^0 \in C^{1,1}(\R)$ such that
\begin{enumerate}[label=(\roman*)]
  \item \label{a:v0e:x0e} $\bx_\e^0 \to \bx^0$ as $\e \to 0$,
  \item $\| \phi_\e^0 \|_\infty = o(1)$ as $\e \to 0$,
  \item \label{a:v0e:IC} for all $x \in \R$
  \begin{equation} \label{v0e}
  v_\e^0 (x) = \sum_{i = 1}^N u \Big( \frac{x - x_{\e,i}^0}\e ; b_i \Big) + \phi_\e^0(x).
\end{equation}
\end{enumerate}
\end{ass}
Note that
\[
  v_\e^0 (-\infty) = o(1)
  \quad \text{and}
  \quad v_\e^0 (\infty) = B + o(1)
  \quad \text{as } \e \to 0,
\]
where
\[
  B := \sum_{i = 1}^N b_i \in \Z
\]
is the number of positively oriented particles minus the number of negatively oriented particles. 

Finally, we remark on the spatial scaling of \eqref{HJe:formal} is a rescaling of the original Peierls-Nabarro model. In the original model the distance between the atoms at rest is $1$ and the distance between dislocations is of order $\frac1\e$. Then, the period of $W$ matches with the atomic distance, and $u$ describes a typical phase transition of $v_\e$ (with layer thickness of order $1$) without the need to rescale by $\e$ as currently done in \eqref{v0e}. However, our aim is to connect to the discrete dislocation dynamics in terms of $x_i(t)$, which is most naturally done when the dislocation distance is of order $1$. 
\smallskip 

Next we describe the discrete dislocation dynamics model in full detail. It is given by the system of ODEs
\begin{equation} \label{PN} \tag{P$_N$}
   \left\{ \begin{aligned}
     &\frac{dx_i}{dt} = c_0 \sum_{ j \in S_t \setminus \{i\} } \frac{b_i b_j}{x_i - x_j}
     && t\in(0,\infty), \ i = 1,\ldots, N \\
     &+ \text{annihilation upon collision}
     && \\
     &\bx(0) = \bx^0, \qquad S_0 = \{1, \ldots, N\},
   \end{aligned} \right.
\end{equation}
where
\beq\label{c0}c_0=\left(\int_\R (u')^2\right)^{-1}\eeq
is a mobility constant (or drag coefficient) and $u$ is the solution of \eqref{u}. 
A proper definition and well-posedness of this particle system is given in \cite{VanMeursPeletierPozarXX}; see also Definition \ref{d:PN:sol} and Proposition \ref{p:vMPP} below. Here, we give a formal description of \eqref{PN}. $N \in \N$ is the number of particles,  $\bx^0 \in \Omega^N$ is an ordered list of initial positions, $\bx = (x_1, \ldots, x_N) \in \R^N$ are the particle positions and $S_t$ is the index set of the surviving particles, i.e.\ the particles which are not annihilated up to and including time $t$. Figure \ref{fig:trajs} illustrates the  dynamics. The dynamics can be thought of as the overdamped limit of positively and negatively charged particles interacting on the real line by the Coulomb potential. Since the Coulomb potential is singular, particle collisions typically happen in finite time and with unbounded particle velocities. The additional feature of \eqref{PN} to this dynamics is the collision rule. This rule states that when two particles of opposite sign collide, they are both taken out (annihilated) from the system. It turns out that more than two particles can collide at the same time-space point, but only if their signs are alternating (when ordered from left to right or vice versa). The collision rule is such that at any multiple-particle collision of $m$ particles, all particles annihilate if $m$ is even and precisely one particle survives if $m$ is odd. The orientation of the surviving particle equals the sum of the orientations of all the colliding particles. The role of the index set $S_t$ is simply to keep track of the surviving particles. Two basic properties of the solution $\bx$ to \eqref{PN} are that the particle positions remain strictly ordered in time (i.e.\ $\{ x_i(t) \}_{i \in S_t} \in \Omega^{\# S_t}$) and that the sum of all orientations is conserved, i.e.
\[
  B = \sum_{i \in S_t} b_i
  \qquad \text{for all } t \geq 0.
\]
\smallskip

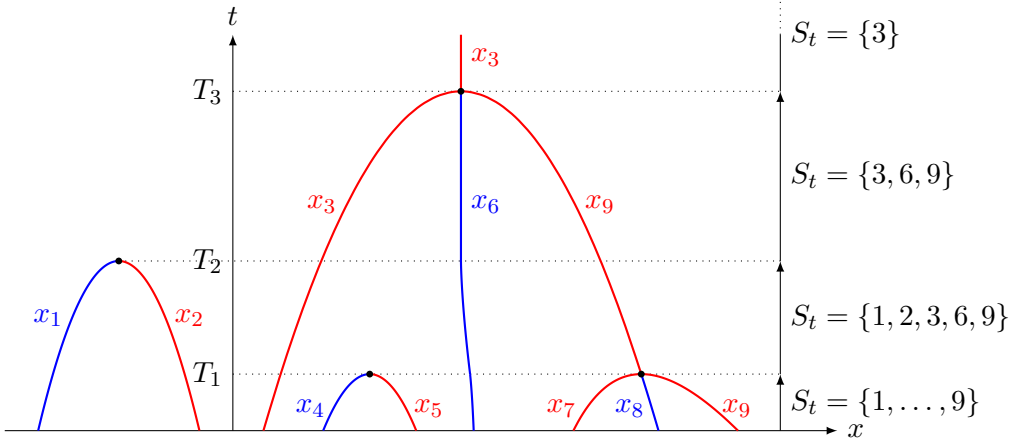
\begin{figure}[ht] 
\centering
\begin{tikzpicture}[scale=1.5, >= latex]
\def \sqtwo {1.414}
\def \rr {0.03}

\draw[->] (0,0) -- (0,3.5) node[above] {$t$};
\draw[->] (-2,0) -- (5.3,0) node[right] {$x$};

\draw[dotted] (0,.5) node[left]{$T_1$} -- (4.8,.5);
\draw[dotted] (-1,1.5) -- (0,1.5) node[left]{$T_2$} -- (4.8,1.5);
\draw[dotted] (0,3) node[left]{$T_3$} -- (4.8,3);

\draw[->] (4.8,0) --++ (0,.5) node[midway, right] {$S_t = \{1,\ldots,9\}$}; 
\draw[->] (4.8,.5) --++ (0,1) node[midway, right] {$S_t = \{1,2,3,6,9\}$}; 
\draw[->] (4.8,1.5) --++ (0,1.5) node[midway, right] {$S_t = \{3,6,9\}$}; 
\draw (4.8,3) --++ (0,.5) node[right] {$S_t = \{3\}$}; 
\draw[dotted] (4.8,3.5) --++ (0,.3);

\draw[blue] (-1.4, 1) node[left] {$x_1$};
\draw[red] (-.6, 1) node[right] {$x_2$};
\draw[red] (1, 2) node[left] {$x_3$};
\draw[blue] (.9, .2) node[left] {$x_4$};
\draw[red] (1.5, .2) node[right] {$x_5$};
\draw[blue] (2, 2) node[right] {$x_6$};
\draw[red] (3.1, .2) node[left] {$x_7$};
\draw[blue] (3.7, .2) node[left] {$x_8$};
\draw[red] (4.2, .2) node[right] {$x_9$};
\draw[red] (3, 2) node[right] {$x_9$};
\draw[red] (2, 3.3) node[right] {$x_3$};


\showfig{
\begin{scope}[shift={(2,3)},scale=1] 
    \draw[thick, red] (0,0) -- (0,.5);
    \draw[thick, blue] (0,-1.5) -- (0,0);
    \draw[domain=-1.732:1.581, smooth, thick, red] plot (\x,{-\x*\x});           
    \draw[domain=1.581:1.732, smooth, thick, blue] plot (\x,{-\x*\x});           
    \fill[black] (0,0) circle (\rr); 
\end{scope}

\begin{scope}[shift={(2,1.5)},scale=1, rotate = 270] 
    \draw[domain=0:1, smooth, thick, blue] plot (\x,{.08*\x*sqrt(\x)}); 
    \draw[domain=1:1.5, smooth, thick, blue] plot (\x,{.08*\x*sqrt(\x) - .1*(\x - 1)*sqrt(\x - 1)});         
\end{scope}

\begin{scope}[shift={(3.581,.5)},scale=1] 
    \draw[domain=-.707:.707, smooth, thick, red] plot ({2*exp(-\x/2)-2},{-\x*\x});
    \fill[black] (0,0) circle (\rr);  
\end{scope}

\begin{scope}[shift={(1.2,.5)},scale=1] 
    \draw[domain=0:.408, smooth, thick, red] plot (\x,{-3*\x*\x});
    \draw[domain=-.408:0, smooth, thick, blue] plot (\x,{-3*\x*\x});     
    \fill[black] (0,0) circle (\rr);     
\end{scope}

\begin{scope}[shift={(-1,1.5)},scale=1] 
    \draw[domain=-.707:0, smooth, thick, blue] plot (\x,{-3*\x*\x});  
    \draw[domain=0:.707, smooth, thick, red] plot (\x,{-3*\x*\x}); 
    \fill[black] (0,0) circle (\rr);         
\end{scope}
}
\end{tikzpicture} \\
\caption{A sketch of solution trajectories to \eqref{PN} with $N=9$. Trajectories of particles with positive orientation are colored red; those with negative orientation blue. }
\label{fig:trajs}
\end{figure}

Informally, we want to prove
\begin{equation*} 
  \eqref{HJe:formal} \longrightarrow \eqref{PN}\quad\text{as }\ep\to0.
\end{equation*}
To connect the solution $\bx$ of \eqref{PN} to the solution $v_\e$ of \eqref{HJe:formal}, we set
\begin{equation} \label{v}  
   v(t,x) := \sum_{i \in S_t} b_i \, H \big( x - x_i(t) \big)
   \qquad \text{for all } t \geq 0 \text{ and all } x \in \R, 
\end{equation}   
where $H$ is the Heaviside function. We remark that $v$ can be interpreted as the unique (discontinuous) viscosity solution to a certain Hamilton-Jacobi equation; see  \cite[Prop.\ 4.5]{VanMeursPeletierPozarXX}.

Finally, to state our result,
we denote  for a function $f :  (0,\infty) \times \R \to \R$ the upper semi-continuous envelope by $f^*$ and the lower semi-continuous envelope by $f_*$. Furthermore, for a sequence of functions $f_\e :  (0,\infty) \times \R \to \R$ parametrized by $\e > 0$, we set
\[
   {\limsup}^* f_\e (t,x)
   := \limsup_{ \substack{ s \to t \\ y \to x \\ \e \to 0 } } f_\e (s,y)
   \quad \text{and} \quad
   {\liminf}_* f_\e (t,x)
   := \liminf_{ \substack{ s \to t \\ y \to x \\ \e \to 0 } } f_\e (s,y).
\]
Our main result is the following:
\begin{thm}[Main] \label{t}
Let  $N \in \N$, $\bb \in \{-1, +1\}^N$ and $\bx^0 \in \Omega^N$.
Let $\bx$ be the solution to \eqref{PN} and $v$ be the corresponding step function defined in \eqref{v}. For each $\e > 0$, let $v_\e^0$ satisfy Assumption \ref{a:v0e} with $\bx^0$ and some $\bx_\e^0 \in \Omega^N$, and let $v_\e$ be the solution of \eqref{HJe:formal} subject to the initial condition $v_\e^0$. Then 
\[
  v_*
  \leq {\liminf}_* v_\e
  \leq {\limsup}^* v_\e \leq v^*
  \qquad \text{on } {[0,\infty) \times \R}.
\]
\end{thm}

Note that the convergence statement is an extension of local uniform convergence to functions $v$ that have jump discontinuities.
Indeed, at a point $(t,x)$ where $x\ne x_i(t)$ for any $i$, we have that $v_*=v^*$ and that $v_\e$ converges uniformly to $v$ in a neighborhood of $(t,x)$. 

\medskip

Our proof method for Theorem \ref{t} is based on the methods developed in \cite{GonzalezMonneau12, PatriziValdinoci15,PatriziValdinoci16}. On subsequent time intervals, we explicitly construct and patch together sub- and supersolutions, $\underline v_\e$ and $\o v_\e$ respectively, of \eqref{HJe:formal} such that $\underline v_\e \leq v_\e \leq \o v_\e$ and
 $v_* \leq {\liminf}_* \underline v_\e
  \leq {\limsup}^* \o v_\e \leq v^*$.
These sub- and supersolutions $\underline v_\e$ and $\o v_\e$ are always of the form \eqref{v0e} with $x_{\e,i}^0$ replaced by carefully constructed particle trajectories $x_{\e,i}(t)$ which remain close to $x_i(t)$.

Patching together sub- and supersolutions is necessary around the collision times of $\bx$. To describe this in more detail, let $T_1$ be the first collision time of \eqref{PN}. By the construction developed in \cite{GonzalezMonneau12, PatriziValdinoci15}, we can construct such $\underline v_\e$ and $\o v_\e$ up to $T_1 - \tau_\e$ for some small $\tau_\e > 0$ which vanishes as $\e \to 0$. If at $T_1$ only two particles collide, then the construction in \cite{PatriziValdinoci16} applies (with obvious modifications to the present setting where $N$ is arbitrary) to construct different sub- and supersolutions $\underline v_\e$ and $\o v_\e$ on $[T_1 - \tau_\e, T_1 + \ttau_\e]$ for some $\ttau_\e$ which vanishes as $\e \to 0$. These sub- and supersolutions are such that at $t = T_1 + \ttau_\e$  they are again of the form \eqref{v0e}, but with the two colliding particles removed. Then, we can repeat the construction in \cite{GonzalezMonneau12, PatriziValdinoci15} again for the $N-2$ surviving particles to get close to the second collision event. From here we can iterate the construction up to time $T$, provided that all collisions are between two particles only.

Our proof extends this construction to the general case, in which two difficulties need to be overcome. The first difficulty is that three or more particles can collide at the same time-space point. We treat such collisions by using the recent results in \cite{VanMeursPeletierPozarXX} which state that the orientations of the colliding particles have to be alternating and that \eqref{PN} is stable with respect to perturbations. The second difficulty is that at a single collision time two or more collisions can take place at different locations (see, e.g.\ Figure \ref{fig:trajs} at $T_1$). Since the construction on $\tau_\e$ and $\ttau_\e$ depends on $\bx$ in a neighborhood around $T_1$, we develop a careful modification of the choice of $\tau_\e$ and $\ttau_\e$ which allows for multiple collisions.

\subsection{Discussion} 

We mention three merits of Theorem \ref{t}. First, with
Theorem \ref{t} it follows from the literature that discrete dislocation dynamics emerges as the dilute dislocation limit from the fully atomistic Frenkel-Kontorova model, i.e.\ the limit in which the number of atoms in between two neighboring dislocations diverges to $\infty$. Indeed, in \cite{fino} the connection is made between the  Frenkel-Kontorova model and the Peierls-Nabarro model, and this connection shows that the number of atoms between neighboring dislocations is of order $\frac1\e$.

The second merit is that, as expected from \cite{GonzalezMonneau12,PatriziValdinoci16}, the mobility constant $c_0$ in the discrete dislocation dynamics model in \eqref{PN} does not change beyond collisions. This is an important observation, because the value of $c_0$ represents the macroscopic influence of the atomistic interaction potential $W$; see \eqref{c0} and \eqref{u}.

The third merit is that Theorem \ref{t} unifies the literature on the limit $N \to \infty$. Indeed, in \cite{VanMeursPeletierPozarXX} it is shown that limit of \eqref{PN} in terms of the function $v$ in \eqref{v}, when rescaling time and space in terms of $N$, is given by
\begin{equation} \label{HJ:formal:ov} \tag{HJ}
  \partial_t v =c_0|\partial_x v|\, \cI[v].
\end{equation}
A rigorous meaning to this PDE was developed in \cite{BKM10} in terms of viscosity solutions. The PDE in \eqref{HJ:formal:ov} also appears in the joint limit $(\e, N) \to (0, \infty)$ for the spatial rescaling in which the typical distance between dislocations is $\frac1N$, see \cite{PatriziSang21,PatriziSang22}. Thus, Theorem \ref{t} shows that both results are consistent in the sense that the sequential limit (first $\e \to 0$, then $N \to \infty$) leads to the same equation as the joint limit. We remark that for densely distributed dislocations, i.e. $\ep=1$ and $N\to+\infty$ a more complicated PDE than \eqref{HJ:formal:ov} appears in the limit; see \cite{MonneauPat12, MonneauPat12b,PatriziValdinoci15b}.
\smallskip

Finally, we mention an interesting extension of the current work. Instead of the half Laplacian $\cI$, any fractional Laplacian $-(-\Delta)^{\frac s{2}}$ with $s \in (0,2)$ can be considered. In \eqref{PN} this corresponds to the particle interaction force $f(x) = C_s x/|x|^{s+1}$ for a certain constant $C_s > 0$. Such generalized setting is considered in \cite{PatriziValdinoci15,PatriziValdinoci16}. The first step for this extension is to extend the result in \cite{VanMeursPeletierPozarXX} on \eqref{PN} to the interaction force $f(x) = C_s x/|x|^{s+1}$; this  work is in progress by the first author.

\subsection{Organization of the paper}
The paper is organized as follows. Section \ref{preliminarysec} contains the preliminaries. 
Section \ref{s:pf} is devoted to the proof of our main result, Theorem \ref{t}.  In Section \ref{s:pf:outline}  we explain  the strategy of the proof. 
In Section \ref{s:pf:0} we establish some further preliminary results. The proof of Theorem \ref{t} before the first collision time is given in Section \ref{s:pf:1}.
The proof of the theorem after the first collision time is given in Sections \ref{s:pf:15}, \ref{s:pf:2} and \ref{s:pf:3}, which consider, respectively, the case of a single collision at a single point, the case of a multiple collision at a single point, and the general case. 

\section{Preliminaries}\label{preliminarysec}


\paragraph{Notation.} For a real-valued sequence $\alpha_\e$ parametrized by $\e > 0$, we abbreviate ``$\alpha_\e \to 0$ as $\e \to 0$" by ``$\alpha_\e = o_\e(1)$". Unless mentioned otherwise, this convergence is understood to be uniform in all other variables such as $t$ and $x$.
We use $C > 0$ as a generic constant independent from the important variables. It may change value from display to display. For $a \in \R$ and $r > 0$, we introduce the one-dimensional balls
\[
  B(a,r) := (a-r, a+r) \subset \R.
\]
We also use
\[
  a \vee b := \max \{a,b\}, \qquad
  a \wedge b := \min \{a,b\}.
\]
For a function $f : \R \to \R^d$, we denote the one-sided limits at $x \in \R$ (if they exist) by
\[
  f(x+) := \lim_{y \downarrow x} f(y)
  \quad \text{and} \quad
  f(x-) := \lim_{y \uparrow x} f(y).
\]
 Finally,  we set
\begin{equation} \label{QT:Q}
  Q := (0,\infty) \times \R.
\end{equation}
\paragraph{Properties of $\cI$.} 
For the reader's convenience, we recall that for any $\alpha, \beta, a, b, x \in \R$ and any functions $v, w : \R \to \R$ regular enough, the Laplace operator $\cI$ defined in \eqref{cI} satisfies
\begin{align*}
  \cI[\alpha v + \beta w] &= \alpha \cI[ v ] + \beta \cI[w] 
  &&\text{linearity,} \\
  \cI[\alpha] &= 0
  &&\text{invariance to constants,} \\
  \cI[v( \, \cdot + b)](x) &= \cI[v](x + b) 
  &&\text{translation invariance,}\\
  \cI[v( a \, \cdot \, )](x) &= |a| \cI[v](ax)
  &&\text{scaling,}\\
  \| \cI[v] \|_{C(\R)} &\leq 4 \| v \|_{C^{1,1}(\R)}
  &&\text{uniform bound.}
\end{align*}

\paragraph{Properties of $u$.}

We recall that $\beta \in (0,1)$ is the constant in \eqref{Wass} related to the regularity of $W$. We further set
 \beqs 
 \alpha:=W''(0)>0.
 \eeqs
The following properties of $u$ are established in 
\cite[Theorem 1.2, Lemma 2.3]{CabreSolaMorales05}. 
 
\begin{lem}[Properties of $u$]
\label{l:u} The solution $u$ of \eqref{u} satisfies $u\in C^{2,\beta}(\R)$. Furthermore, there exists a constant $K_1 >0$ such that
\begin{equation*} 
\left|u(x)-H(x)+\frac{1}{\alpha \pi
x}\right|\leq \frac{K_1}{x^2} \quad\text{for }|x|\geq 1.
\end{equation*}
\end{lem} 
\paragraph{The corrector $\psi$.}
As in \cite{GonzalezMonneau12} we introduce the function  $\psi$ to be the solution of
\begin{equation}\label{psi}
\left\{ \begin{aligned}
  \cI[\psi] 
  &= W''(u)\psi + \frac{W''(u) - W''(0)}\alpha  + u' 
  && \text{in } \R \\
\lim_{|x| \rightarrow \infty} \psi(x)
&=0. &&
\end{aligned} \right.
\end{equation} 
We will use $\psi$ as an $O(\e)$ correction to construct sub and  supersolutions to \eqref{HJe:formal}. 
For a detailed heuristic motivation of
equation \eqref{psi}
see  \cite[Section 3.1]{GonzalezMonneau12}.

We recall from \cite[Thm 3.2]{GonzalezMonneau12} and
\cite[Lemma 3.2]{MonneauPat12b} the following decay estimate 
on $\psi$: 
\begin{lem}[Properties of $\psi$]
\label{l:psi}
There exists a unique solution $\psi \in C^1(\R) \cap W^{1,\infty}(\R)$ to \eqref{psi}. Furthermore, there exist constants $K_2 \in \R$
and $K_3 > 0$  such that
\begin{equation*} 
\left|\psi(x)-\frac{K_2}{
x}\right|\leq\frac{K_3}{x^2} \quad\text{for }|x|\geq 1.
\end{equation*}
\end{lem}

\paragraph{Existence and comparison principle of \eqref{HJe:formal}.} 

Let us denote by $USC_b(Q)$ (recall $Q$ from \eqref{QT:Q}), respectively $LSC_b(Q)$, the set of bounded, upper semicontinuous, respectively lower semicontinuous, functions on
$Q$. Set 
$C_b(\o Q):=USC_b(Q) \cap LSC_b(Q)$. 

The definition of viscosity solutions and the following comparison theorem are given  in \cite[Definition 2.1 and Theorem 3.1]{JakoKar} for more general parabolic integro-PDEs. 

\begin{prop}[Comparison Principle for \eqref{HJe:formal}]\label{comparisonuep}  Let
 $u\in USC_b(\overline{Q})$ 
and $v\in LSC_b(\overline{Q})$ be respectively viscosity  sub and supersolution of \eqref{HJe:formal} in $Q$,
such that $u(0,x)\leq v(0,x)$ for all $x\in\R$. 
Then $u\leq v$ in  $Q$. 
\end{prop}

\begin{prop}[Existence for \eqref{HJe:formal}-\eqref{v0e}]\label{existuep}For $\ep>0$ there exists a viscosity solution $v_{\ep}\in C_b(\o Q)$ of
\eqref{HJe:formal}-\eqref{v0e}. Moreover, there exists $C>0$ such that $\| v_\e \|_{C(Q)} \leq C$ uniformly in $\e$.
\end{prop}


\begin{proof}
We can construct a
solution by Perron's method if we can construct sub and supersolutions
of \eqref{HJe:formal} which are equal to $ v_\e^0 (x)$ at $t= 0$. Since by Lemma \ref{l:u} $ v_\e^0 \in C^{1,1}(\R)$, 
 the two functions
$u^{\pm}(t,x):= v_\e^0 \pm C_\e t$ with  
$$C_\e\geq \frac{4}{\e}\| v_\e^0 \|_{C^{1,1}(\R)}+\frac{1}{\e^2}\|W'\|_\infty$$
are respectively a super and a
subsolution of \eqref{HJe:formal}. Moreover $u^+(0,x)=u^-(0,x)= v_\e^0 (x)$.

Finally, by comparison with constant solutions of \eqref{HJe:formal} with value in $\Z$ we conclude from $\| v_\e^0 \|_{C(\R)} \leq C$ uniformly in $\e$ that $\| v_\e \|_{C(Q)} \leq C$ uniformly in $\e$.
\end{proof}

\paragraph{The solution of \eqref{PN} and its properties.} 

For the following definition we recall $\Omega^N$ from \eqref{Om}. 
\begin{defn}[Solution to \eqref{PN}] \label{d:PN:sol}
Given $\bb \in \{-1, +1\}^N$ and $\bx^0 \in \Omega^N$, we call $\bx := \{ x_i : [0,T_{k_i}] \to \R \}_{i=1}^N$ a solution to \eqref{PN} with respect to the initial condition $\bx^0$ if for any $T>0$ there exist indices $k_i \in \{1, \ldots, K+1\}$ and collision times $\cT := \{ T_1, \ldots, T_K \}$ with $0 =: T_0 < T_1 < \ldots < T_K \leq T_{K+1} := T$  such that (set $S_t := \{ i : T_{k_i} > t \}$):
\begin{enumerate}[label=(\roman*)]
  \item $x_i \in C^1( (0,T_{k_i}) \setminus \cT ) \cup C^0([0,T_{k_i}])$ for all $1 \leq i \leq N$,
  \item $x_i(0) = x_i^0$ for all $1 \leq i \leq N$,
  \item \label{d:PN:sol:order} $\{ x_i(t) \}_{i \in S_t} \in \Omega^{\# S_t}$ for all $t > 0$,
  \item $\{ x_i(t) \}_{i \in S_t}$ satisfies the ODE in \eqref{PN} on $(T_k,T_{k+1})$ for all $0 \leq k \leq K$,
  \item \label{d:PN:sol:annih} for each $t \in \cT$, the set $S_{t-} \setminus S_{t}$ is non-empty and can be written as a disjoint union of sets $\{i,j\}$ for which $x_i(t) = x_j(t)$ and $b_i b_j = -1$.
\end{enumerate}
\end{defn}

Definition \ref{d:PN:sol} is technical; we refer to Figure \ref{fig:trajs} for a visual representation, and give several remarks below to parse it:
\begin{itemize}
  \item $T_{k_i}$ is the annihilation time of $x_i$. It is part of the solution concept.
  \item The set of surviving particles $S_t$ is, as a function of $t$, constant on $[T_k, T_{k+1})$ for each $k$, and $S_s \supset S_t$ for all $s < t$. 
  \item Thanks to the strict ordering required in \ref{d:PN:sol:order}, the right-hand side of \eqref{PN} is defined at each $t >0$. Note that \ref{d:PN:sol:order} puts a strong restriction on the solution at collisions. Indeed, at collisions the strict ordering breaks down. Hence, for \ref{d:PN:sol:order} to hold at collisions, it is therefore required that at most one of the colliding particles survives.
  \item Condition \ref{d:PN:sol:annih} takes the opposite role of (iii). Whereas (iii) is easy to satisfy if we are free to annihilate particles at will, \ref{d:PN:sol:annih} put restrictions on the situation at which particles can be annihilated. In \ref{d:PN:sol:annih}, $S_{t-} \setminus S_{t}$ is the set of particles which get annihilated at $t$. The word ``disjoint" prevents that too many particles are annihilated. Without this wording, it would be possible for the example in Figure \ref{fig:trajs} that at $t = T_3$ all three particles annihilate.
  \item Note the subtle difference between \textit{collision} times and \textit{annihilation} times. An annihilation time (in Definition \ref{d:PN:sol} denoted by $T_{k_i}$) is a property of a particle $x_i$; it is the time at which $x_i$ is annihilated, i.e.\ when $i$ is taken away from $S_t$. A collision time (in Definition \ref{d:PN:sol} denoted by $T_k$) is a time point at which two or more particles collide. Therefore, an annihilation time is always a collision time, but even when a particle $x_i$ collides (necessarily at a collision time), it need not be the annihilation time for $x_i$ (see, for instance, the surviving particles in Figure \ref{fig:trajs} at the 3-particle collisions). Other than in Definition \ref{d:PN:sol} we do not make use of annihilation times.
\end{itemize}

Whenever convenient, we follow the convention to extend the domain of $x_i$ beyond its annihilation time $T_{k_i}$ to $\infty$ by setting $x_i(t) := x_i(T_{k_i})$ for all $t > T_{k_i}$. 

We do not expect uniqueness for the solution $\bx$, because at a collision between an odd number of particles the index of the particle which survives can be chosen freely (see, e.g., Figure \ref{fig:trajs}, where we can construct a different solution by taking $S_t = \{9\}$ for $t \geq T_3$.).
We say that a solution $\bx$ to \eqref{PN} with respect to $\bx^0$ is unique if the union of the graphs of $x_i |_{[0, T_{k_i}]}$ are unique.

\begin{prop}[{\cite[Thm.\ 2.4]{VanMeursPeletierPozarXX}}] \label{p:vMPP}
For  $N \geq 1$, $\bb \in \{-1, +1\}^N$ and $\bx^0 \in \Omega^N$, there exists a unique solution $\bx$ to \eqref{PN}. Moreover, at any time-space collision point $(T_k,y)$ the orientations of the colliding particles are alternating.
\end{prop}

In addition to Proposition \ref{p:vMPP}, \cite[Theorem 2.4(vi)]{VanMeursPeletierPozarXX} provides a stability result of \eqref{PN} with respect to perturbations of the initial data. For our purposes, we require a modification of this stability result. We state it below in Lemma \ref{l:obx}.

\section{Proof of Theorem \ref{t}}
\label{s:pf}

The second inequality in Theorem \ref{t} is obvious, and the first and third can be proven in a similar manner. Therefore, it is sufficient to prove 
\begin{equation} \label{pf:zz}
  {\limsup}^* v_\e \leq v^*
   \qquad \text{on } {[0,T) \times \R}
\end{equation}
for all $T > 0$. In the remainder we take $T$ arbitrary, and assume it to be large enough whenever convenient.

\subsection{Outline of the proof of \eqref{pf:zz} }
\label{s:pf:outline}
The proof resembles those from \cite{GonzalezMonneau12, PatriziValdinoci15,PatriziValdinoci16}. We apply an induction argument over the finitely many collision times of the solution $\bx$ to \eqref{PN}. At each induction step, we construct and tie together several supersolutions of \eqref{HJe:formal} on subsequent time intervals (see Figure \ref{fig:timeline} for an overview), which all lie above $v_\e$ within an $o_\e(1)$ height distance. Most of these supersolutions are constructed from the Gonzalez-Monneau Lemma (Lemma \ref{l:vebar}) for different choices of the parameters and initial conditions. These supersolutions are of the form
\begin{equation} \label{supersol:prototype} 
\sum_{i = 1}^N u\left(\displaystyle \frac{x - x_{\e,i}(t)}{\e} ; b_i\right) +  o_\e(1),
\end{equation}
where $u$ is the phase-transition profile defined by \eqref{u}, $o_\e(1)$ is a perturbation term which may depend on $t,x$, and $x_{\e,i}(t)$ are carefully constructed perturbations of $x_i(t)$.

\begin{figure}[ht]
\centering
\begin{tikzpicture}[scale=1, >= latex]
\def \a {0.1}

\draw[->] (0,0) -- (11,0) node[below] {$t$};
\foreach \x in {0,1,3,4,5,8,9,10}
  \draw (\x,-\a) --++ (0,2*\a);
\foreach \x in {0,4,9}
  \draw[dotted] (\x,0) --++ (0,1);

\draw (0,-\a) node[below]{$0$};  
\draw (1,-\a) node[below]{$\tT_0$}; 
\draw (3,-\a) node[below]{$\o T_1$}; 
\draw (4,-\a) node[below]{$\hat T_1$}; 
\draw (5,-\a) node[below]{$\tT_1$}; 
\draw (8,-\a) node[below]{$\o T_2$}; 
\draw (9,-\a) node[below]{$\hat T_2$}; 
\draw (10,-\a) node[below]{$\tT_2$}; 

\draw (.5,0) node[above]{$\tv_\e$};
\draw (2,0) node[above]{$\o v_\e$};
\draw (3.5,0) node[above]{$\hat v_\e^3$};

\draw[<->] (4,.3) --++ (1,0) node[midway, above]{$o_\e(1)$};
\draw[<->] (5,.3) --++ (3,0) node[midway, above]{$O(1)$};
\draw[<->] (8,.3) --++ (1,0) node[midway, above]{$o_\e(1)$};
  
\end{tikzpicture} \\
\caption{Time axis tessellated in intervals on which we construct the supersolutions $\tv_\e$, $\o v_\e$ and $\hat v_\e^3$. The dotted lines indicate different steps in the induction argument. For each $k$, $\o T_k \leq \hat T_k < \tT_k$ are all asymptotically close to the collision times $T_k$ of $\bx$ as $\e \to 0$, but none need to be equal to $T_k$.}
\label{fig:timeline}
\end{figure}
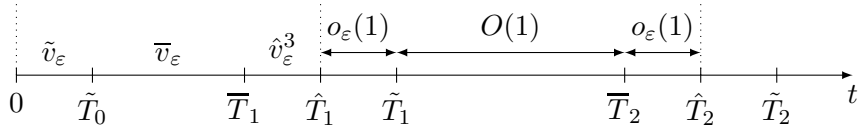

Next we describe the outline of the construction in more detail. We start with some preparation. Let
$ 0 < T_1 < T_2 < \ldots < T_K$
be all the collision times of the solution $\bx$ to \eqref{PN}. We assume  for convenience that $T > T_K$ and set $T_0 := 0$ and $T_{K+1} := T$. For the sake of simplicity, we assume in the outline below that $b_1 = 1 = -b_2$ and that at $T_1$ only the two particles $x_1$ and $x_2$ collide.

The first supersolution which we construct is $\tv_\e$ (defined precisely later in \eqref{pf:za}); see Figures \ref{fig:ve:0} and \ref{fig:ve:tT} for a sketch. More precisely, we require
\begin{subequations} \label{pf:zw} 
\begin{align} \label{pf:zw:IC}
\tv_\e(0, \cdot) &
  \geq v_\e^0
  &&\text{on } \R \\\label{pf:zw:supersol}
  \tv_\e 
  &\geq v_\e
  && \text{on } [0, \tT_0] \times \R
\end{align}
\end{subequations}
for all $\e > 0$ small enough and for some $0 < \tT_0 = o_\e(1)$. 
The motivation of this construction is to control the perturbation $\phi_\e^0$ to the initial datum. As one would expect from Allen-Cahn type equations, the contribution of the term $\phi_\e^0$ will become $o(\e)$ after a time interval of size $o_\e(1)$. We will construct $\tv_\e$ of the form \eqref{supersol:prototype} such that it is $o(\e)$ close to integer values away from transition layers at $t = \tT_0$. 

\begin{figure}[ht]
\centering
\begin{tikzpicture}[scale=1.8, >= latex]
\draw[->] (0,-.25) -- (0,1.25);
\draw[->] (-.25,0) -- (5.25,0) node[right] {$x$};
\draw (0,1) node[left]{$1$} --++ (5.25,0)  node[anchor = south east] {$t = 0$};

\showfig{
\draw[domain=-.25:5.25, smooth, thick, red] plot (\x,{ .5*( atan( 10*(\x - 1) )/90 + 1 ) - .5*( atan( 10*(\x - 4) )/90 + 1 ) + .15*sin(\x * 1000) });
\draw[domain=-.25:5.25, smooth, thick, blue] plot (\x,{ .5*( atan( 10*(\x - .75) )/90 + 1 ) - .5*( atan( 10*(\x - 4.25) )/90 + 1 ) + .2 });
}

\draw[<-] (2.28,1) --++ (0,-.3) node[below] {$o_\e(1)$};
\draw[<-] (2.28,1.2) --++ (0,.3);
\draw[<-] (.75,.4) --++ (-.3,0);
\draw[<-] (1,.4) --++ (.3,0) node [right]{$o_\e(1)$};

\draw[red] (1,0) node[below]{$x_1$} --++ (0,.55);
\draw[blue] (.75,0) node[below]{$\tx_1$} --++ (0,.65);
\draw[red] (4,0) node[below]{$x_2$} --++ (0,.45);
\draw[blue] (4.25,0) node[below]{$\tx_2$} --++ (0,.65);
\draw[red] (3.75, .5) node{$v_\e$};
\draw[blue] (4.5, .7) node{$\tv_\e$};

\end{tikzpicture} \\
\caption{A sketch of $v_\e^0$ and $\tv_\e(0, \cdot)$. }
\label{fig:ve:0}
\end{figure}
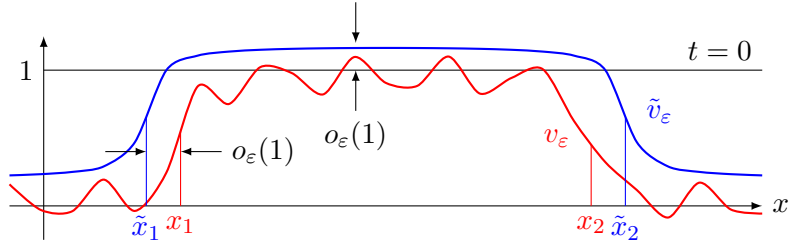

\begin{figure}[ht]
\centering
\begin{tikzpicture}[scale=1.8, >= latex]
\draw[->] (0,-.25) -- (0,1.25);
\draw[->] (-.25,0) -- (5.25,0) node[right] {$x$};
\draw (0,1) node[left]{$1$} --++ (5.25,0)  node[anchor = south east] {$t = \tT_0$};

\showfig{
\draw[domain=-.25:5.25, smooth, thick, blue] plot (\x,{ .5*( atan( 10*(\x - 1) )/90 + 1 ) - .5*( atan( 10*(\x - 4) )/90 + 1 ) + .08 });
\draw[domain=-.25:5.25, smooth, thick, red] plot (\x,{ .5*( atan( 10*(\x - .75) )/90 + 1 ) - .5*( atan( 10*(\x - 4.25) )/90 + 1 ) + .12 });
}

\draw[<-] (2.5,1) --++ (0,-.3) node[below] {$o(\e)$};
\draw[<-] (2.5,1.1) --++ (0,.3);
\draw[<-] (.75,.4) --++ (-.3,0);
\draw[<-] (1,.4) --++ (.3,0) node [right]{$o_\e(1)$};

\draw[blue] (1,0) node[below]{$\tx_1$} --++ (0,.55);
\draw[red] (.75,0) node[below]{$\o x_1$} --++ (0,.6);
\draw[blue] (4,0) node[below]{$\tx_2$} --++ (0,.55);
\draw[red] (4.25,0) node[below]{$\o x_2$} --++ (0,.6);
\draw[blue] (3.75, .5) node{$\tv_\e$};
\draw[red] (4.5, .7) node{$\o v_\e$};

\end{tikzpicture} \\
\caption{A sketch of $\tv_\e$ and $\o v_\e$ at $t = \tT_0$. $\tT_0 = o_\e(1)$ is chosen to be large enough such that $\tv_\e$ is $o(\e)$-close to integer-values away from the transitions around $\tx_1$ and $\tx_2$.}
\label{fig:ve:tT}
\end{figure}
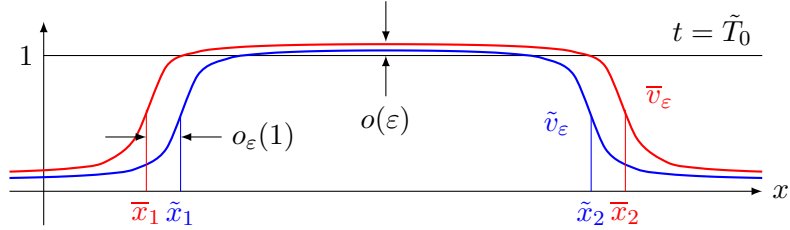

With the upper bound $v_\e(\tT_0, x) \leq \tv_\e (\tT_0, x)$, we have enough control to construct a different supersolution $\o v_\e$  which remains close enough to $v_\e$ almost up to $T_1$. The supersolution $\o v_\e$ is of the form \eqref{supersol:prototype},is  defined precisely in \eqref{pf:yz}, and is illustrated partially in Figures \ref{fig:ve:tT} and \ref{fig:ve:oT}. More precisely, we construct $\o v_\e$ such that
\begin{subequations} \label{pf:zy}
\begin{align} \label{pf:zy:IC}
  \o v_\e(\tT_0, \cdot) &
  \geq \tv_\e (\tT_0, \cdot)
  &&\text{on } \R \\\label{pf:zy:supersol}
  \o v_\e 
  &\geq v_\e  
  && \text{on } [\tT_0, \o T_1] \times \R
\end{align}
for all $\e$ small enough, and
\begin{equation} \label{pf:zy:limp}
  {\limsup}^* \o v_\e \leq v^*
  \qquad \text{on } (0, T_1) \times \R.
\end{equation}
\end{subequations}
Inequalities \eqref{pf:zy:supersol} and \eqref{pf:zy:limp} imply the desired \eqref{pf:zz} for all $t\in(0,T_1)$ and all $x\in\R$.

\begin{figure}[ht]
\centering
\begin{tikzpicture}[scale=1.8, >= latex]
\draw[->] (0,-.25) -- (0,1.25);
\draw[->] (-.25,0) -- (6.5,0) node[right] {$x$};
\draw (0,1) node[left]{$1$} --++ (6.5,0)  node[anchor = south east] {$t = \o T_1$};

\showfig{
\draw[domain=-.25:2.5, smooth, thick, red] plot (\x,{ .5*( atan( 10*(\x - 2) )/90 + 1 ) - .5*( atan( 10*(\x - 3) )/90 + 1 ) + .08 });
\draw[domain=2.5:5.25, smooth, thick, red] plot (\x,{ .5*( atan( 10*(\x - 2) )/90 + 1 ) - .5*( atan( 10*(\x - 3) )/90 + 1 ) + .08 });
\draw[domain=5.25:6.5, smooth, thick, red] plot (\x,{ .5*( atan( 10*(\x - 2) )/90 + 1 ) - .5*( atan( 10*(\x - 3) )/90 + 1 ) + .08 });
\draw[domain=-.25:6.5, smooth, thick, blue] plot (\x,{ .5*( atan( 10*(\x - 1) )/90 + 1 ) - .5*( atan( 10*(\x - 5.5) )/90 + 1 ) + .08 });
}

\draw[<->] (1,.3) --++ (1,0) node[midway, above]{$\Theta_\e$};
\draw[<->] (2,.3) --++ (1,0) node[midway, above]{$\Theta_\e$};
\draw[<->] (3,.3) --++ (2.5,0) node[midway, above]{$L \Theta_\e$};

\draw[blue] (1,0) node[below]{$\hat x_1$} --++ (0,.55);
\draw[red] (2,0) node[below]{$\o x_1$} --++ (0,.55);
\draw[blue] (5.5,0) node[below]{$\hat x_2$} --++ (0,.55);
\draw[red] (3,0) node[below]{$\o x_2$} --++ (0,.55);
\draw[red] (3.25, .5) node{$\o v_\e$};
\draw[blue] (5.75, .5) node{$\hat v_\e^1$};

\end{tikzpicture} \\
\caption{A sketch of $\o v_\e$ and $\hat v_\e^1$ (defined later in \eqref{pf:zf}) at $t = \o T_1$. $\o T_1$ is such that $\o x_2 - \o x_1 =: \Theta_\e = o_\e(1)$, and $L > 1$ controls the asymmetry in $\hat v_\e^1$.}
\label{fig:ve:oT}
\end{figure}
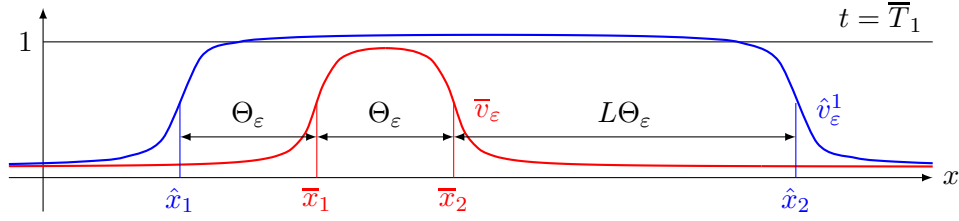
  
We have insufficient control on $\o v_\e$ to ensure that it remains above $v_\e$ during the collision. To go beyond the collision, we construct from $\o T_0$ onwards a different supersolution $\hat v_\e^3$. This construction is ingenious; $\hat v_\e^3$ will be the minimum of two supersolutions $\hat v_\e^1, \hat v_\e^2$ of  \eqref{HJe:formal} of the form \eqref{supersol:prototype} which do \emph{not} resemble a collision. The idea of this construction is illustrated in Figures \ref{fig:ve:oT} and \ref{fig:ve:hatT}. We put the colliding particles further apart in an asymmetric manner such that $\hat v_\e^1$ and $\hat v_\e^2$ can evolve a little while longer (until $\hat T_1 > \o T_1$ with $\hat T_1 = T_1 + o_\e(1)$) before a collision happens. More precisely, we require for $\ell = 1,2$ that
\begin{subequations} \label{pf:zx}
\begin{align} \label{pf:zx:IC}
  \hat v_\e^\ell (\o T_1, \cdot) 
  &\geq \o v_\e (\o T_1, \cdot)
  && \text{on } \R \\ \label{pf:zx:supersol}
  \hat v_\e^\ell 
  &\geq v_\e
  && \text{on } [\o T_1, \hat T_1] \times \R 
\end{align}
\end{subequations}
for all $\e$ small enough. We take the parameter $L$ in Figure \ref{fig:ve:oT} large enough such that the expected collision in $\hat v_\e^1$ will happen in space to the right of the expected collision in $v_\e$, whereas the collision in $\hat v_\e^2$ will happen instead to the left. Then, the profile  
\[
  \hat v_\e^3 = \min \{ \hat v_\e^1, \hat v_\e^2 \}
\]
at $\hat T_1$ turns out to be close to $v_\e$ after the annihilation has happened; see Figure \ref{fig:ve:hatT}. Note that, due to taking the minimum, $\hat v_\e^3$ need not be of the form \eqref{supersol:prototype}. However, at $t = \hat T_1$ it can be bounded from above by a similar expression in which the first two particles are removed:
\begin{equation} \label{pf:wy} 
\hat v_\e^3(\hat{T}_1,x)
\leq \sum_{i = 3}^N u\left(\displaystyle \frac{x - y_i}\e; b_i\right) + o_\e(1),
\end{equation}
where $y_i = x_i(T_1) + o_\e(1)$.

The upper bound in \eqref{pf:wy} is sufficient for iterating the construction above by starting with another $\tv_\e$ defined on $[\hat T_1, \tT_1]$. This completes one induction step. The induction stops after $K+1$ with $\o T_{K+1} = T$. By \eqref{pf:zy:limp}, this proves the desired inequality \eqref{pf:zz} on $(0, T] \setminus \{ T_1, \ldots, T_K \}$.

\begin{figure}[ht]
\centering
\begin{tikzpicture}[scale=1.8, >= latex]
\draw[->] (0,-.25) -- (0,1.25);
\draw[->] (-.25,0) -- (5.25,0) node[right] {$x$};
\draw (0,1) node[left]{$1$} --++ (5.25,0)  node[anchor = south east] {$t = \hat T_1$};

\showfig{
\draw[domain=-.25:2.5, smooth, blue] plot (\x,{ .5*( atan( 10*(\x - 2.75) )/90 + 1 ) - .5*( atan( 10*(\x - 3.75) )/90 + 1 ) + .08 });
\draw[domain=2.5:4, smooth, blue] plot (\x,{ .5*( atan( 10*(\x - 2.75) )/90 + 1 ) - .5*( atan( 10*(\x - 3.75) )/90 + 1 ) + .08 });
\draw[domain=4:5.25, smooth, blue] plot (\x,{ .5*( atan( 10*(\x - 2.75) )/90 + 1 ) - .5*( atan( 10*(\x - 3.75) )/90 + 1 ) + .08 });

\draw[domain=-.25:1, smooth, blue] plot (\x,{ .5*( atan( 10*(\x - 1.25) )/90 + 1 ) - .5*( atan( 10*(\x - 2.25) )/90 + 1 ) + .08 });
\draw[domain=1:2.5, smooth, blue] plot (\x,{ .5*( atan( 10*(\x - 1.25) )/90 + 1 ) - .5*( atan( 10*(\x - 2.25) )/90 + 1 ) + .08 });
\draw[domain=2.5:5.25, smooth, blue] plot (\x,{ .5*( atan( 10*(\x - 1.25) )/90 + 1 ) - .5*( atan( 10*(\x - 2.25) )/90 + 1 ) + .08 });

\draw[domain=-.25:2.5, smooth, very thick, red, dotted] plot (\x,{ .5*( atan( 10*(\x - 2.75) )/90 + 1 ) - .5*( atan( 10*(\x - 3.75) )/90 + 1 ) + .08 });
\draw[domain=2.5:5.25, smooth, very thick, red, dotted] plot (\x,{ .5*( atan( 10*(\x - 1.25) )/90 + 1 ) - .5*( atan( 10*(\x - 2.25) )/90 + 1 ) + .08 });
}

\draw[<-] (2.5,0) --++ (0,-.3); 
\draw[<-] (2.5,.2) --++ (0,.3) node[above] {$o_\e(1)$};

\draw[blue] (4, .5) node{$\hat v_\e^1$};
\draw[blue] (1, .5) node{$\hat v_\e^2$};
\draw[red] (3.25, .3) node{$\hat v_\e^3$};

\end{tikzpicture} \\
\caption{A sketch of $\hat v_\e^\ell$ at $t = \hat T_1$. $\hat v_\e^2$ is constructed as in Figure \ref{fig:ve:oT} with the asymmetry to the left. $\hat T_1$ is chosen to be large enough such that $\hat v_\e^3$ remains $o_\e(1)$-close to $0$-level set.}
\label{fig:ve:hatT}
\end{figure}
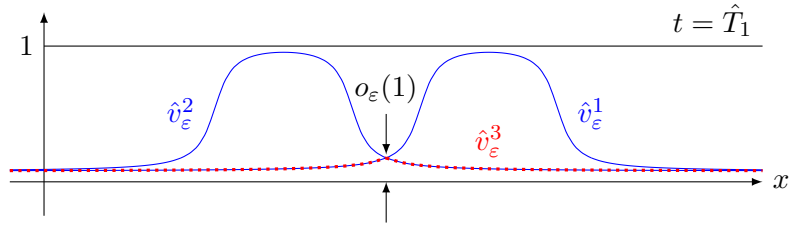

It is still left to show that \eqref{pf:zz} holds at $T_k$ for $k = 0, \ldots, K$. By the induction argument, it is sufficient to focus on $T_1$. For an approximating sequence $t_\e \to T_1$, $t_\e$ can belong to any of the four intervals
\[
  [\tT_0, \o T_1), \quad
  [\o T_1, \hat T_1), \quad
  [\hat T_1, \tT_1), \quad
  [\tT_1, \o T_2). 
\] 
Depending on which interval it belongs to, we will prove \eqref{pf:zz} at $T_1$ for the corresponding upper bound on $v_\e$ which we have constructed on that interval.
\smallskip

Finally, we comment on the treatment of different types of particle collisions. First, if $b_1 = -1$ instead of $b_1 = +1$, then the construction can be simplified. At $\o T_1$, instead of Figure \ref{fig:ve:oT}, $\o v_\e$ is as in Figure \ref{fig:ve:oT:inversed}. Then, the dip in the profile of $\o v_\e(\o T_1, \cdot)$ can simply be removed; we can immediately proceed with the construction of $\tv_\e$ on $[\hat T_1, \tT_1]$ with $\hat T_1 = \o T_1$.
Second, in the case where more than two particles collide, we apply a similar construction up to $t = \o T_1$. Then, before constructing $\hat v_\e^\ell$ which satisfy \eqref{pf:zx:IC}, we first bound $\o v_\e(\o T_1, x)$ from above by a profile $\hat v_\e^0(x)$ which fits to the case of no collision or a simple collision as considered above. Then, similar arguments as used in this section apply to continue the construction of supersolutions beyond $t = \o T_1$.
Third, we show that our construction is robust to the case in which different annihilation events happen at the same time. This completes the outline of the proof of \eqref{pf:zz}. 
\medskip 

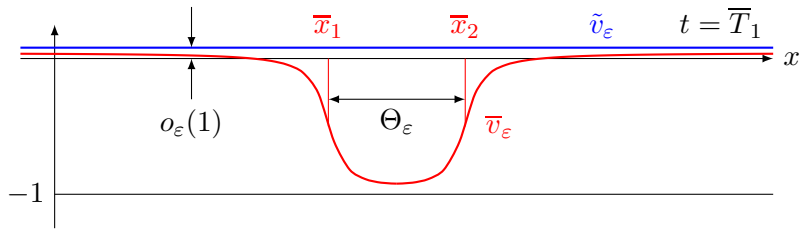
\begin{figure}[ht]
\centering
\begin{tikzpicture}[scale=1.8, >= latex]
\draw[->] (0,-1.25) -- (0,.25);
\draw[->] (-.25,0) -- (5.25,0) node[right] {$x$};
\draw (0,-1) node[left]{$-1$} --++ (5.25,0);
\draw (5.25,.08) node[anchor = south east] {$t = \o T_1$};

\showfig{
\draw[domain=-.25:2.5, smooth, thick, red] plot (\x,{ -.55*( atan( 10*(\x - 2) )/90 + 1 ) + .55*( atan( 10*(\x - 3) )/90 + 1 ) + .04 });
\draw[domain=2.5:5.25, smooth, thick, red] plot (\x,{ -.55*( atan( 10*(\x - 2) )/90 + 1 ) + .55*( atan( 10*(\x - 3) )/90 + 1 ) + .04 });
\draw[thick, blue] (-.25,.08) -- (5.25,.08);
}

\draw[<->] (2,-.3) --++ (1,0) node[midway, below]{$\Theta_\e$};

\draw[red] (2,0) --++ (0,-.5);
\draw[red] (3,0) --++ (0,-.5);
\draw[red] (2,.08) node[above] {$\o x_1$};
\draw[red] (3,.08) node[above] {$\o x_2$};
\draw[red] (3.25, -.5) node{$\o v_\e$};
\draw[blue] (4, .08) node[above]{$\tv_\e$};

\draw[<-] (1,0) --++ (0,-.3) node[below] {$o_\e(1)$}; 
\draw[<-] (1,.08) --++ (0,.3);

\end{tikzpicture} \\
\caption{A sketch of $\o v_\e$ and $\tv_\e$ (defined later in \eqref{pf:zf}) at $t = \o T_1$ in the case $b_1 = -1$.}
\label{fig:ve:oT:inversed}
\end{figure}

The remainder of this section entails the rigorous proof of \eqref{pf:zz}. It is organized as follows. In the preparatory Section \ref{s:pf:0} we build the key lemmas on which the construction relies. In Section \ref{s:pf:1} we construct $\tv_\e$ and $\o v_\e$ and prove \eqref{pf:zz} on $(0,T_1)$, without any assumptions on the type of collision. In Section \ref{s:pf:15} we construct $\hat v_\e^\ell$ for (simple) collisions at which precisely two particles collide at $T_1$. We also prove how to continue the construction up to some $\o T_2 = T_2 + o_\e(1)$, and prove \eqref{pf:zz} at $t = T_1$. Then, in Section \ref{s:pf:2} we extend this construction to the case in which the particle collision at $T_1$ contains more than two particles, and in Section \ref{s:pf:3} we deal with multiple collision events at separated locations. Finally, in Section \ref{s:pf:3} we make the induction argument precise. 

\subsection{Preparation}
\label{s:pf:0}

\paragraph{We may take $c_0 = 1$.}
Here we prove that without loss of generality we may assume that 
\[
  c_0 = 1.
\]
We recall that $c_0$ is the mobility parameter in \eqref{PN}. Given $W$, the value of $c_0 > 0$ follows from \eqref{c0} and \eqref{u}. It is easy to check that the spatially rescaled version
\[
  \hat v (t, \hat x) := v^\e (t, \sqrt{c_0} \hat x)
\]
satisfies \eqref{HJe:formal} with $\e$ replaced by $\hat \e := \sqrt{c_0} \e$ and $W$ replaced by $\hat W (v) := c_0 W(v)$. 
Since $\hat W$ satisfies \eqref{Wass}, equation \eqref{u} with $W$ replaced by $\hat W$ has a unique solution $\hat u$. It is easy to check that $\hat u(\hat x) = u (c_0 \hat x)$ 
and that
\[
  \hat c_0 
  := \left(\int_\R (\hat u')^2\right)^{-1} 
  = \frac1{c_0} \left(\int_\R (u')^2\right)^{-1}
  = 1.
\]
Hence, translating the statement of Theorem \ref{t} to the setting denoted with hats we obtain an equivalent statement with $\hat c_0 = 1$. 

\paragraph{Explicit expression for $v^*$.} 
With $v$ given by \eqref{v} with respect to the solution $\bx$ to \eqref{PN}, its upper semi-continuous envelope can written as
\begin{equation} \label{v*} 
   v^*(t,x) := \sum_{i \in S_t} (b_i H)^* \big( x - x_i(t) \big) + \chi(t,x)
   \qquad \text{for all } t \geq 0 \text{ and all } x \in \R, 
\end{equation}
where the indicator function $\chi(t,x) \in \{0,1\}$ equals $1$ if and only if at $(t,x)$ an even number of particles collide for which the leftmost particle has positive orientation. The expression \eqref{v*} is easy to validate when $(t,x)$ is not a collision point (for a visual interpretation, consider Figure \ref{fig:trajs}, and note that $v$ is constant on each of the regions separated by the trajectories). If it is a collision point, then it follows from the alternating orientations of the  colliding particles (see Proposition \ref{p:vMPP}) that $v$ in any small enough neighborhood $\mathcal N$ around $(t,x)$ attains precisely two values ($k$ and $k+1$ for some $k \in \Z$). Then, $v^*(t,x) = k+1$. However, the sum in \eqref{v*} only selects the highest value of $v$ in $\mathcal N_+ := \mathcal N \cap (t,\infty) \times \R$. For those collisions for which $v |_{\mathcal N_+} \equiv k$, the function $\chi$ fixes the mismatch between the sum and $v^*(t,x)$.

From Lemma \ref{l:u} we observe that the supersolutions from Section \ref{s:pf:outline} of the form \eqref{supersol:prototype} converge to a sum of Heaviside functions as $\e \to 0$. In preparation for proving \eqref{pf:zz} with $v^*$ as in \eqref{v*}, the following lemma provides a convenient upper bound.

\begin{lem}[Relation between $u$ and $H$]
\label{l:u:to:H} Let $x \in \R$ and $b \in \{-1,+1\}$. For any $x_\e \to x$ as $\e \to 0$,
\begin{equation}\label{u:to:H}
\limsup_{\e \to 0} u \left( \frac{x_\e}\e; b \right) \leq (bH)^*(x). 
\end{equation}
\end{lem}

\begin{proof}
Since the proof for $b = -1$ is similar to that of $b = 1$, we focus on the case $b=1$. Then, \eqref{u:to:H} reads
\[
  \limsup_{\e \to 0} u \left( \frac{x_\e}\e \right) \leq H^*(x).
\]
Since $u \leq 1$, we may assume that $x < 0$. Then, $\frac{x_\e}\e \to -\infty$ as $\e \to 0$, and thus $u \left( \frac{x_\e}\e \right) \to 0 = H^*(x)$ as $\e \to 0$.
\end{proof}




\paragraph{Super and subsolutions to \eqref{HJe:formal}}
In this subsection we recall from \cite{GonzalezMonneau12,PatriziValdinoci15} how  to construct supersolutions of  \eqref{HJe:formal} by using $u$ and $\psi$. We use this construction to build $\o v_\e$, $\hat v_\e^1$ and $\hat v_\e^2$ introduced in Section \ref{s:pf:outline}. Several of the properties of this construction (see Lemma \ref{l:obx} below) are new.



The construction starts from a perturbed version of \eqref{PN}. Given an external, constant force $\o \sigma \in\R$, we define $\o \bx (t) := (\xs_1(t),\ldots,\xs_N(t))$ to be the solution of the following perturbed version of \eqref{PN} with initial points $\o \bx^0 \in \Omega^N$: for $i=1,\ldots,N$ 
\beq\label{oPN}\begin{cases} \dot{\xs}_i = 
\displaystyle\sum_{ \substack{ j = 1 \\ j \neq i} }^N 
\displaystyle\frac{b_i b_j }{ \xs_i-\xs_j}- b_i \o \sigma  &\text{in }(0, \osT_1)\\
 \xs_i(0) = \o x_i^0,
\end{cases}\eeq 
where $\osT_1$ is the first collision time of \eqref{oPN}. If no collision happens in finite time, then we take $\osT_1 = T$. Note that the external force $\o \sigma$, when positive, pushes positive particles to the left and negative particles to the right. In \eqref{oPN} we are not interested in going beyond collisions. Hence, there is no need to specify a collision rule and no need to keep track of an index set of surviving particles. 

 For later use, let $\theta_i, \theta : [0,T_1] \to \R$ and $\o\theta_i, \o\theta : [0,\osT_1] \to \R$ for $i = 1,\ldots,N-1$ be defined by
\begin{align} \label{theta:defs}
  \theta_i &:= x_{i+1} - x_i,
  & \theta &:= \min_{1 \leq i \leq N-1} x_{i+1} - x_i, \\\notag
  \o\theta &:= \o x_{i+1} - \o x_i,
  & \o\theta &:= \min_{1 \leq i \leq N-1} \o x_{i+1} - \o x_i.
\end{align}

Lemma \ref{l:obx} below states the connection between \eqref{oPN} and \eqref{PN}.

\begin{lemma}[Perturbation of \eqref{PN}] \label{l:obx}
Let $\bx$ be the solution of \eqref{PN} with initial condition $\bx^0  \in \Omega^N$.
Let $\o\bx$ be the solution of \eqref{oPN} with initial condition  $\o \bx^0 =  \bx^0+\o\eta\in \Omega^N$ for some $\o\eta\in\R^N$. 
Let $T_1$ and $\osT_1$ be the first collision time respectively of \eqref{PN} and \eqref{oPN}.  
Then 
 for any $\tau \in (0, T_1)$ fixed and any $\o t = \o t(\o \sigma, \o\eta)$ such that $\o t\leq \osT_1$ and  $\o t \to T_1$ as $\o \sigma, |\o\eta| \to 0$, there holds
\begin{align}
  \| \o \bx - \bx \|_{C([0, T_1 - \tau])} &\to 0, \label{uniformconvergencexbartox} \\
   \osT_1 &\to T_1, \label{l:obx:eqn} \\
  \o \bx(\overline{t}) &\to \bx(T_1) \label{xbartetoxoft}
\end{align}
as $\o \sigma, |\o\eta| \to 0$. 
\end{lemma} 
\begin{proof}
For convenience we parametrize the sequences $\o \sigma$ and $\o \eta$ by $\e$ such that ``$ \o \sigma, |\o\eta|\to0$" can be summarized as $\e\to0$. 


The limit  \eqref{uniformconvergencexbartox} follows from standard ODE theory. Indeed, on $[0, T_1 - \tau]$ the particles $x_i$ remain separated, and thus the right-hand side in \eqref{PN} remains smooth and bounded in a neighborhood of the trajectories. Since the perturbation in \eqref{oPN} vanishes as $\e \to 0$, \eqref{uniformconvergencexbartox} follows. Moreover, from the particle separation of $x_i$ and \eqref{uniformconvergencexbartox} we also have that the particles $\o x_i$ remain separated, i.e.\ for any $\tau \in (0, T_1)$ there exists $c > 0$ such that 
$$\min_{1 \leq i \leq N-1}\min_{t\in [0,T_1-\tau]} \xs_{i+1}(t)-\xs_i(t) > c,$$ for $\ep $ small enough, 
which implies that $ \osT_1>T_1-\tau$. By the arbitrariness of  $\tau$, we  infer that 
\beqs 
\liminf_{\e\to0}  \osT_1\geq T_1.
\eeqs
Then, to prove \eqref{l:obx:eqn}, it suffices to show
\beq\label{l:obx:eqn_sup} \limsup_{\e\to0}  \osT_1\leq T_1.\eeq
We remark that, in the special case where $T_1 = T$ (i.e.\ no collisions happen), the argument above holds for $\tau = 0$, and then \eqref{l:obx:eqn} and \eqref{xbartetoxoft} follow from \eqref{uniformconvergencexbartox}.
\smallskip

In the generic case $T_1 < T$, it is left to prove \eqref{l:obx:eqn_sup} and \eqref{xbartetoxoft}. We prove both statements together from the argument that follows. As preparation, for any particle $x_i$, we set $I_i = \{k, k+1, \ldots,\ell\}$ as the set of indices of all particles $x_j$ (including $j=i$) that collide with $x_i$ at time $T_1$. If  $x_i$ does not collide, then we take $I_i = \{i\}$.
 
Let $\delta > 0$ be small and arbitrary. By the continuity of $\bx$ there exists $\tau_0=\tau_0(\delta)>0$ such that for any $0<\tau<\tau_0$
\[
  | x_i(T_1 - \tau) - x_i(T_1) | < \frac\delta4
  \quad\text{for any } 1 \leq i \leq N.
\]
Then, by \eqref{uniformconvergencexbartox}, 
there exists   $\ep_0=\ep_0(\tau,\delta)$ such that 
for any $0<\ep<\ep_0$,
\beq\label{initialQi}|\xs_i(T_1-\tau)-x_i(T_1)|<\frac{\delta}{2}\quad\text{for any } 1 \leq i \leq N.\eeq

We are going to show that, for any $i$, the  trajectory of $ \xs_i$ remains in the  time-space box $(T_1 - \tau, \osT_1 ) \times
 B( x_i(T_1), \delta) $. 
Let  $T_e$ be the first time at which any particle $x_i$ exits $B( x_i(T_1),\delta)$, that is the first time bigger than $T_1-\tau$ such that either $T_e = \osT_1$ or  
$|\xs_i(T_e) - x_i(T_1) | = \delta$  for some $i \in \{1,\ldots,N \}$. Note that
\beq\label{insideQi}|\xs_i(t)-x_i(T_1)| < \delta\quad \text{for any }t\in (T_1-\tau, T_e)\text{ and any } 1 \leq i \leq N.\eeq
Let $2\rho$ be the shortest distance between any two particles $x_i$ which do not collide with each other at time $T_1$, that is, 
\begin{equation*} 
  \rho := \frac12 \min_{1 \leq i \leq N} \min_{j \notin I_i} |x_i(T_1) - x_j(T_1)| > 0.
\end{equation*}  
By \eqref{insideQi} we may assume that $\delta$ is small enough with respect to $\rho$ such that
\begin{equation} \label{pf:wk}  
    \min_{1 \leq i \leq N} \min_{j \notin I_i} |\xs_i (t)- \xs_j(t) |> \rho \quad \text{for any }t\in (T_1-\tau, T_e).
\end{equation} 

Next, take any $i$ for which $I_i = \{k, k+1, \ldots,\ell\}$ with $\ell>k$. Then, by \eqref{pf:wk} and for $\delta$ small enough with respect to $\rho$, the estimates in the proof of \cite[Theorem 2.4(vi)]{VanMeursPeletierPozarXX} on the ODE in \eqref{oPN} reveal that 
\beq\label{Ddot2} \dot \xs_k> 0\quad\text{and}\quad \dot \xs_\ell< 0\quad\text{ in } (T_1-\tau,  T_e) \eeq
and that $\o D:= \xs_\ell-\xs_k$ satisfies  
\beq\label{theta1} \dot{\o D}\leq -\frac{1}{\o D}+\sum_{j\in I_i^c}\frac{b_\ell b_j}{\xs_\ell-\xs_j}-\sum_{j\in I_i^c}\frac{b_k b_j}{\xs_k-\xs_j}+2|\o\sigma|\quad\text{ in } (T_1-\tau, T_e ).\eeq
Then, from \eqref{theta1} and \eqref{pf:wk} we obtain 
\beq\label{Ddot1} \dot{\o D}\leq -\frac{1}{\o D}+ \underbrace{ 2\frac{N}{\rho}+2 }_{=: C_{\rho,N}} \quad\text{ in } (T_1-\tau, T_e).\eeq
Moreover, by \eqref{initialQi}
\beq\label{Dinitial}\o D(T_1-\tau)\le\delta.\eeq
By \eqref{Ddot1} and \eqref{Dinitial} we infer that $\o D$ is decreasing in $(T_1-\tau, T_e)$ for $\delta$ small enough with respect to $C_{\rho,N}$. In fact, taking $\delta$ smaller if necessary, we have
\beqs \frac12 \frac d{dt} (\o D^2) 
= \dot{\o D}\,\o D\leq -1+ C_{\rho,N} \o D\leq   -1+ C_{\rho,N} \o D(T_1-\tau)\leq -1+ C_{\rho,N} \delta \leq -\frac12 \quad\text{ in } (T_1-\tau,  T_e).\eeqs 
Integrating from $T_1-\tau$ to $T_e$ we obtain
\beqs 0\le \o D^2(T_e)\leq \o D^2(T_1-\tau)-(T_e-T_1+\tau)\leq \delta^2-(T_e-T_1+\tau).\eeqs
This implies that 
\beq\label{exittimeest} T_e\leq T_1-\tau+\delta^2.\eeq

Next we show that $T_e = \osT_1$, i.e.\ no particle $\o x_i$ exits their respective time-space box before $t = \osT_1$. Suppose instead that $T_e < \osT_1$ and let $\xs_i$ be a particle which exits its box at $t = T_e$, that is  $|\xs_i(T_e) - x_i(T_1) | = \delta$. From \eqref{initialQi} and \eqref{Ddot2} it follows that $\# I_i$ cannot be greater than or equal to $2$, and thus $I_i = \{i\}$. Then, by \eqref{oPN} and 
\eqref{pf:wk}   we have that 
\beqs 
|\dot \xs_i|\leq \left|\sum_{j\neq i}\frac{b_ib_j}{\xs_i-\xs_j}-b_i\o\sigma\right|\leq  \frac{N}{\rho}+1 = \frac12 C_{\rho,N}
\quad\text{for all }t\in (T_1-\tau, T_e).
\eeqs
Combining this with \eqref{initialQi} and 
\eqref{exittimeest}, we have
\beqs \delta 
=|x_i(T_1) - \xs_i(T_e)|
\leq |x_i(T_1) - \xs_i(T_1-\tau)| + \frac12 C_{\rho,N} (T_e-T_1+\tau) 
\leq  \frac{\delta}{2} + \frac12 C_{\rho,N} \delta^2,\eeqs 
which is a contradiction for $\delta$ small enough with respect to $C_{\rho,N}$. We conclude that $T_e = \osT_1$. 

Finally, the desired statements \eqref{l:obx:eqn_sup} and \eqref{xbartetoxoft} follow from $T_e = \osT_1$ and the fact that $\delta$ can be chosen arbitrarily small. Indeed, \eqref{l:obx:eqn_sup} follows from \eqref{exittimeest}, and \eqref{xbartetoxoft} follows from \eqref{l:obx:eqn_sup} and \eqref{insideQi}.
\end{proof}

Next we use \eqref{oPN} to construct supersolutions of \eqref{HJe:formal}. We set
\begin{equation} \label{phib}
  \psi(x; b) := b \psi (bx)
  \quad \text{for } x \in \R \text{ and } b = \pm 1.
\end{equation}
Note that the dependence on $b$ is different from that in $u$:
\begin{equation} \label{ub}
  u(x; b) = u (bx) + \frac{b-1}2.
\end{equation}
Using these relations, we obtain for a constant $C \in \R$ that
\begin{equation*}
  (u + C \psi)(x; b) 
  = \begin{cases}
    u(x) + C \psi(x)
    &\text{if } b = +1 \\
    u(-x) - 1 - C \psi(-x)
    &\text{if } b = -1;
  \end{cases}
\end{equation*}
we will often use this expression.

Setting
\beq\label{oci}\cs_i(t):= \dot{\xs}_i(t),\quad i=1,\ldots,N,\eeq 
we define
\beq\label{ove}\begin{split}
\vs_\ep(t,x)
&:=\sum_{i=1}^N ( u - \ep\cs_i(t)\psi ) \left(\displaystyle \frac{x-\xs_i(t)}{\ep} ; b_i \right)+ \ep \frac{\o \sigma - \delta_\e}\alpha,\end{split}
\eeq
for some $\delta_\e > 0$, where we recall that $\alpha = W''(0) > 0$.
Under the appropriate choice of the parameters,
Lemma \ref{l:vebar} below states that $\vs_\ep$ is a supersolution of \eqref{HJe:formal}. 


\begin{lemma}[Gonzalez--Monneau] \label{l:vebar} There exist $\ep_0>0$ and $\theta_\ep,\,\delta_\ep>0$ with 
\beq\label{thetaepthetaprop} \theta_\ep,\,\delta_\ep,\,\ep\theta_\ep^{-2}=o(1)\quad\text{as }\ep\to0\eeq
such that for any $\o \bx^0 \in \Omega^N$, $\ep \in (0, \ep_0)$, $\o \sigma \geq \delta_\e$ the function $\vs_\ep$ defined in \eqref{ove} from the solution $\o \bx$ of \eqref{oPN} is a supersolution of \eqref{HJe:formal} on
\[
  \{ t\in(0, \osT_1) \mid \o \theta_i \geq \e \} \times \R.
\]
\end{lemma} 

\begin{proof}
See the proofs of Proposition 5.3 and Lemma 5.4 in \cite{PatriziValdinoci16}.
\end{proof}

\begin{rem} \label{r:thetae:dele} 
The $\e$-dependent constants $\theta_\ep,\,\delta_\ep$ can be constructed explicitly. A possible choice is $\theta_\e = \e^\gamma$ for any $\gamma < \frac12$, and $\delta_\e = C \ep^\alpha$ for some $C, \alpha > 0$, where $\alpha$ depends on $\gamma$ and $C$ depends on $N$ and  both $\phi$ and $\psi$. Hence, $\theta_\ep,\,\delta_\ep$ can be chosen such that they only depend on $W$ and $N$ and not on any of the parameters $ \o \sigma, \o \bx^0$ in \eqref{oPN}. We treat $\theta_\ep,\,\delta_\ep$ as given, $\e$-dependent constants in the remainder of the proof of Theorem \ref{t}.
\end{rem}



\paragraph{Patching together different supersolutions.} The functions $\tv_\e$, $\o v_\e$ and $\hat v_\e^\ell$ from Section \ref{s:pf:outline} are of the form \eqref{ove}. We need to patch them together at $\o T_k$, $\hat T_k$ and $\tT_k$ such that \eqref{pf:zy:IC} and \eqref{pf:zx:IC} hold (see Figures \ref{fig:ve:tT} and \ref{fig:ve:oT}). With this aim we establish the following lemma.

\begin{lem} [Bounding supersolutions by other supersolutions] \label{l:u:epsi:new}
For any $M > 0$ there exist $K = K(M) \geq 0$ such that for all $\vartheta_\e > 0$ there exists $\e_0 > 0$ such that for all $c, c' \in \R$ and all $\e, z > 0$, if
\[
  \e < \e_0, \quad
  \vartheta_\e, \frac\e{\vartheta_\e} = o_\e(1), \quad
  |c|, |c'| \leq \frac{M}{\vartheta_\e}, \quad
  z \geq \frac{\vartheta_\e}\e,
\]
then
\begin{equation} \label{pf:xy}
  \left(u - \e |c \psi| \right)(y+z)
  - \left(u + \e |c' \psi| \right)(y)
  \geq - K(M) (|c| + |c'|) \frac{\e^2}{\vartheta_\ep}
  \quad \text{for all } y \in \R. 
\end{equation}
\end{lem}

\begin{proof}
We set $\vartheta := \vartheta_\e$. We recall $\alpha = W''(0) > 0$ and take
\begin{equation*}
  \kappa := \min \left\{ \frac1{10}, \frac1{8 \alpha \pi M \|\psi\|_\infty } \right\} > 0
\end{equation*}
as an $\e$-independent constant. We split 3 cases depending on $y$.
\smallskip

\emph{Case 1: }$y \notin B(-z, \kappa \frac\vartheta\e) \cup B(0, \kappa \frac\vartheta\e)$.
In this case, we don't need the precise expression for $\kappa$. Since $|y|, |y+z| \geq \kappa \frac\vartheta\e$, we obtain from Lemma \ref{l:psi} that
\begin{equation} \label{pf:xu}  
  \max \left\{ |\psi|(y+z), |\psi|(y) \right\}
  \leq \frac{C}\kappa \frac\e\vartheta
\end{equation}
for some constant $C > 0$ and $\e$ small enough. Then, simply using that $u$ is increasing, we obtain \eqref{pf:xy} by
\[
\left(u - \e |c \psi| \right)(y+z)
  - \left(u + \e |c' \psi| \right)(y)
  > -  \frac{C}\kappa (|c| + |c'|) \frac{\e^2}\vartheta.
\]
\smallskip

\emph{Case 2: } $y \in B(-z, \kappa \frac\vartheta\e)$. 
Using that $u$ is increasing and Lemma \ref{l:u}, we obtain for all $\e$ small enough that both
\[
  u(y+z) 
  > u \left( - \kappa \frac\vartheta\e \right)
  \geq \frac1{\alpha \pi \kappa} \frac\e\vartheta - K_1 \left(\frac{\e}{\kappa \vartheta}\right)^2
  \geq \frac1{2 \alpha \pi \kappa} \frac\e\vartheta
\]
and 
\[
  u(y) 
  \leq u \left( \kappa \frac\vartheta\e - z \right)
  \leq u \left( -(1-\kappa) \frac\vartheta\e \right)
  \leq \frac2{\alpha \pi (1-\kappa)} \frac\e\vartheta.
\]
Then, using that $\kappa \in (0, \frac1{10})$,
\[
  u(y+z) - u(y)
  \geq \frac1{2 \alpha \pi} \frac{1 - 5\kappa}{\kappa (1-\kappa)} \frac\e\vartheta
  \geq \frac1{4 \alpha \pi \kappa} \frac\e\vartheta. 
\]
This yields
\begin{align*}
\left(u - \e |c \psi| \right)(y+z)
  - \left(u + \e |c' \psi| \right)(y)
  &\geq \frac1{4 \alpha \pi \kappa} \frac\e\vartheta
  - \e (|c| + |c'|) \| \psi \|_\infty \\
  &\geq \Big( \frac1{4 \alpha \pi \kappa} 
  - 2 M \| \psi \|_\infty \Big) \frac\e\vartheta,
\end{align*}
which is positive since $\frac1\kappa \geq 8 \alpha \pi M \| \psi \|_\infty$. Hence, if $y \in B(-z, \kappa \frac\vartheta\e)$, then \eqref{pf:xy} holds for any $K_4 \geq 0$.
\smallskip

\emph{Case 3: } $y \in B(0, \kappa \frac\vartheta\e)$. By the symmetry in the estimate on $u$ in Lemma \ref{l:u}, this case can be treated similarly as in Case 2.
\end{proof}

The following lemma will be used to remove certain dipoles from supersolutions. An example of such a dipole is illustrated in Figure \ref{fig:ve:oT:inversed}.

\begin{lem} [Removing certain dipoles] \label{l:u:epsi:dipole}
Under the same conditions as Lemma \ref{l:u:epsi:new}
\begin{equation*}
  (u - \e c \psi)(y+z; -1)
  + (u - \e c' \psi)(y; +1)
  \leq K(M) \frac{\e^2}{\vartheta_\e^2}
  \quad \text{for all } y \in \R.
\end{equation*}
\end{lem}

\begin{proof}
We proceed similarly as in the proof of Lemma \ref{l:u:epsi:new}. We set $\vartheta := \vartheta_\e$. By the bound on $c, c'$, we rewrite and estimate the left-hand side as
\begin{align*}
  (u - \e c \psi)(y+z; -1)
  + (u - \e c' \psi)(y; +1)
  &= (u + \e c \psi)(-y-z) - 1
  + (u - \e c' \psi)(y) \\
  &\leq \left(u + \frac\e\vartheta M |\psi| \right)(-y-z) - 1
  + \left(u + \frac\e\vartheta M |\psi| \right)(y).
\end{align*}
Hence, it is sufficient to show that
\begin{equation} \label{pf:xx}
  \left(u + \frac\e\vartheta M |\psi| \right)(-y-z) - 1
  + \left(u + \frac\e\vartheta M |\psi| \right)(y)
  \leq K_4 (|c| + |c'|) \frac{\e^2}{\vartheta}
  \quad \text{for all } y \in \R.
\end{equation}
To prove this, we take
\begin{equation*}
  \kappa :=  \frac1{4 \alpha \pi M \|\psi\|_\infty + 6 } > 0
\end{equation*}
and split 3 cases depending on $y$.
\smallskip

\emph{Case 1: } $y \leq -z - \kappa \frac\vartheta\e$. Since $|y|, |y+z| \geq \kappa \frac\vartheta\e$, we may use the bound on $\psi$ in \eqref{pf:xu}
for $\e$ small enough. Then, using Lemma \ref{l:u},
\begin{align*}
  u(-y-z) 
  &\leq 1 + \frac1{ \alpha \pi (y+z) } + \frac{K_1}{(y+z)^2}
  \leq 1 + \frac1{ \alpha \pi (y+z) } + \frac{K_1}{\kappa^2} \frac{\e^2}{\vartheta^2}  \\
  u(y) 
  &\leq - \frac1{ \alpha \pi y } + \frac{K_1}{y^2}
  \leq - \frac1{ \alpha \pi y } + K_1 \frac{\e^2}{\vartheta^2}.
\end{align*} 
Hence
\[
\left(u + \frac\e\vartheta M |\psi| \right)(-y-z) - 1
  + \left(u + \frac\e\vartheta M |\psi| \right)(y)
  \leq \frac{-z}{ \alpha \pi y (y+z) } + C_M \frac{\e^2}{\vartheta^2}
  \leq 0 + C_M \frac{\e^2}{\vartheta^2}
\]
for some constant $C_M$ which may depend on $M$ but not on $y,z,\e$.
\smallskip 

\emph{Case 2: } $y \in [-z - \kappa \frac\vartheta\e, - \frac z2]$. 
Using that $u$ is increasing and Lemma \ref{l:u}, we obtain
\begin{align*}
  u(-y-z) 
  &\leq u \left( \kappa \frac\vartheta\e \right)
  \leq 1 - \frac1{2 \alpha \pi \kappa} \frac\e\vartheta, \\
  u(y)
  &\leq u \left( - \frac z2 \right)
  \leq \frac2{\alpha \pi z} + \frac{4 K_1}{z^2}
  \leq \frac3{\alpha \pi} \frac\e\vartheta
\end{align*}
for $\e$ small enough. Substituting this into the left-hand side of \eqref{pf:xx} and recalling the expression of $\kappa$,
\[
  \left(u + \frac\e\vartheta M |\psi| \right)(-y-z) - 1
  + \left(u + \frac\e\vartheta M |\psi| \right)(y)
  \leq \left( 6 - \frac1\kappa \right) \frac1{2\alpha \pi} \frac\e\vartheta + 2M \| \psi \|_\infty \frac\e\vartheta
  \leq 0.
\]
\smallskip

\emph{Case 3: } $y > - \frac z2$. By the symmetry in the estimate on $u$ and $\psi$, this case can be treated similarly as in Cases 1 and 2.
\end{proof}

\subsection{Construction of super solutions before collision}
\label{s:pf:1}

In this section we prove \eqref{pf:zz} in $(0,T_1)$.
The proof is a combination of techniques developed in \cite{GonzalezMonneau12,PatriziValdinoci15,PatriziValdinoci16,PatriziValdinoci17}; a few steps simplify thanks to the recently established Proposition \ref{p:vMPP}. Nonetheless, we give a self-contained proof.

\subsubsection{Construction of $\tv_\e$ in $[0, \tT_0]$ }\label{subsubvtildeep}
We  construct a function $\tv_\e$ for which \eqref{pf:zw} holds. Let $\rho_\e := \| \phi_\e^0 \|_\infty$. For parameters $\mu, \varrho_\e > 0$ to be specified later let 
\begin{equation} \label{pf:yv}
  \tilde \bx(t) := \bx_\e^0 - \bb \varrho_\e (1 - e^{- \mu t / \e^2})
\end{equation}
and set 
\begin{equation*} 
  \tv_\e(t, x) := \sum_{i = 1}^N u\left(\displaystyle \frac{x - \tx_i (t)}{\ep} ; b_i\right) + \rho_\e e^{- \mu t / \e^2}.
\end{equation*}
Recalling that $u$ is increasing, we obtain that \eqref{pf:zw:IC} holds for any choice of $\mu, \varrho_\e > 0$. Then, by \cite[Lemma 4.1]{PatriziValdinoci17} there exist $\mu, \, \varrho_\e, \, \ttau_\ep>0$ satisfying  
 \beq\label{pf:zb} \varrho_\ep = o(1),\quad \ttau_\ep=o_\e(1),\quad \tilde \rho_\ep := \rho_\ep e^{-{\mu \ttau_\ep}/{\ep^{2}}}=o(\e) \eeq 
such that, for $\e$ small enough, $\ttau_\e < T_1/2$ and $\tv_\e$ is a supersolution of \eqref{HJe:formal} on $[0, \ttau_\e]$. Then, by \eqref{pf:zw:IC} the comparison principle (see Proposition \ref{comparisonuep}) implies that \eqref{pf:zw:supersol} holds with 
\begin{equation}\label{Ttildezero} 
\tT_0: = \ttau_\e =o_\ep(1). 
\end{equation}
For later use, we note that 
\begin{equation} \label{pf:za}
  \tv_\e(\tT_0, x) = \sum_{i = 1}^N u\left(\displaystyle \frac{x - \tx_i (\ttau_\e)}{\ep} ; b_i\right) + \tilde \rho_\e 
  \quad \text{for all } x \in \R
\end{equation}
and that by \eqref{pf:yv} and $\bx_\e^0 \to \bx^0$ as $\e \to 0$ we have
\begin{equation} \label{pf:yw}
  \max_{0 \leq t \leq \ttau_\e} \max_{1 \leq i \leq N} | \tx_i(t) - x_i^0 |
  = o_\e(1). 
\end{equation}

\subsubsection{Construction of $\o v_\e$ in $[ \tT_0, \o T_1 ]$}
\label{subsubvepbar}
Next we construct $\o v_\e$ for which \eqref{pf:zy} holds. 
First, we assume that a collision happens, i.e.\ $T_1 < T$, and comment on the  case of no collisions afterwards. Considering the perturbed system \eqref{oPN}, we take as parameters 
\begin{equation} \label{pf:xw}
  \o \sigma := \alpha \frac{\tilde \rho_\e}\e + \alpha C_0 \frac\e{\theta_\e} + \delta_\e = o_\e(1)
  \quad \text{and} \quad 
  \o \bx^0 := \tilde \bx (\ttau_\e) - \theta_\e \bb,
\end{equation}
where $\theta_\e, \delta_\e > 0$ are defined in Lemma \ref{l:vebar} (see Remark \ref{r:thetae:dele}), 
$\tilde \rho_\e = o(\e)$ (recall \eqref{pf:zb}) and the $\e$-independent constant $C_0 > 0$ is chosen later. Let $\o \bx$ be the solution of \eqref{oPN}, $\osT_1$ be the time of the first collision of $\o \bx$ and, similar to \eqref{oci}-\eqref{ove},
\begin{equation} \label{pf:yz}
  \vs_\ep(t,x)
:=\sum_{i=1}^N ( u - \ep\cs_i(t - \tT_0)\psi ) \left(\displaystyle \frac{x-\xs_i(t - \tT_0)}{\ep} ; b_i \right) + \ep \frac{ \o \sigma - \delta_\e}\alpha
\end{equation} 
 for $t \in [\tT_0, \tT_0 + \osT_1)$ and $x \in \R$. 
 
Next we prove \eqref{pf:zy:IC}. 
We start with bounding $\o c_i(0) = \dot{\o x}_i(0)$. We observe from \eqref{pf:yw} and \eqref{pf:xw} that
\[
  |\o x_i(0) - x_i^0|
  \leq |\o x_i(0) - \tx_i (\ttau_\e)| + |\tx_i (\ttau_\e) - x_i^0|
  = \theta_\e + o_\e(1)
  = o_\e (1)
  \quad \text{for } i =1, \ldots, N.
\]
Hence, recalling the definitions of $\theta, \o\theta$ from \eqref{theta:defs},
\[
  \o \theta_i(0)
  \geq \o \theta (0)
  \geq \frac12 \theta (0)
  > 0
\]
for $\e$ small enough. Then, from the right-hand side of the ODE in \eqref{oPN} we observe that $| \o c_i(0) | \leq C_1$ for some constant $C_1 > 0$ independent of $\e$. 

Using this bound and recalling \eqref{pf:za}, \eqref{pf:xw} and \eqref{pf:yz}, we obtain for the left-hand side in \eqref{pf:zy:IC} that
\[
  (\o v_\e - \tv_\e) (\tT_0, x) 
  \geq \sum_{i=1}^N \left( (u - \e C_1 |\psi|) \left( y_i + b_i \frac{\theta_\e}\e ; b_i\right) 
   - u \left( y_i; b_i \right) \right)
   + \e \frac{\o \sigma  - \delta_\e}\alpha - \tilde \rho_\e,
\]
where $y_i := (x - \tx_i(\ttau_\e))/\e \in \R$. Recalling \eqref{phib} and \eqref{ub}, the summand equals
\[
  (u - \e C_1 |\psi|) \left( b_i y_i + \frac{\theta_\e}\e \right) - u \left( b_i y_i \right). 
\]
Then, applying Lemma \ref{l:u:epsi:new} with $\vartheta_\ep = \theta_\e$, $c=C_1$ and $c'=0$, we obtain 
\begin{equation} \label{pf:xv}
  (\o v_\e - \tv_\e) (\tT_0, x) 
  \geq - C_1 N K(C_1) \frac{\e^2}{\theta_\e} 
   + \e \frac{\o \sigma - \delta_\e}\alpha - \tilde \rho_\e
\end{equation}
for all $\e$ small enough. Then, taking $C_0 := C_1 N K(C_1)$ in our choice of $\o \sigma$ in \eqref{pf:xw}, the right-hand side in \eqref{pf:xv} is non-negative. This proves \eqref{pf:zy:IC}.
 
Next we prove \eqref{pf:zy:supersol}. By Lemma \ref{l:vebar}, we have for all $\e$ small enough that $\o v^\e (t + \tT_0, x)$ is a supersolution of \eqref{HJe:formal} on $\cT \times \R$, where
\[
  \cT := \{ t \in [0, \osT_1) : \o \theta (t) \geq \theta_\e \}.
\]
We claim that there exists a $\o \tau_\e  = o_\e(1)$ such that 
\begin{equation} \label{pf:xp} 
  \cT \supset [0, T_1 - \o \tau_\e]
\end{equation}
for $\e$ small enough. Then, \eqref{pf:zy:supersol} follows from \eqref{pf:zy:IC} by the comparison principle (see Proposition \ref{comparisonuep}) with
\begin{equation} \label{pf:yu}  
  \o T_1 := \tT_0 + T_1 - \o \tau_\e =T_1+o_\ep(1).
\end{equation}
Next we prove the claim.
Recall that $\theta \in C([0, T_1])$ satisfies 
\beqs 
\theta(t)>0 \quad\text{for }t\in[0,T_1), \text{ and } \quad \theta(T_1)=0.
\eeqs
Then, by Lemma \ref{l:obx}, we have for $\ep$ small enough that there exists $\tau_\ep>0$ such that 
\beq \label{thetabar>=thetaep}\o\theta(t)>\theta_\ep \quad\text{for any }t<\osT_1-\tau_\ep, \text{ and } \quad \o\theta(\osT_1-\tau_\ep)=\theta_\ep.\eeq
Moreover, since $\theta_\ep=o_\ep(1)$, we have 
\beqs 
\tau_\ep=o_\ep(1).
\eeqs
Hence, \eqref{pf:xp} follows with 
\beq\label{otaudef}\o\tau_\e:=\tau_\e+T_1-\osT_1=o_\ep(1).\eeq
For later use, we observe that  \eqref{thetabar>=thetaep} can be rewritten as 
 \begin{equation} \label{thetabar=thetaepbis} 
 \o\theta (T_1 - \o \tau_\e) =\theta_\ep= o_\e(1),
\end{equation}
and 
from  \eqref{thetabar>=thetaep} 
and the ODE in \eqref{oPN} (recall \eqref{pf:yu}) that 
 \begin{equation} \label{pfxe} 
  \max_{1 \leq i \leq N} \max_{t\in [\tT_0, \o T_1 ]} \o c_i (t - \tT_0)\leq \frac C{\theta_\e}.
\end{equation}

 Next, we   prove \eqref{pf:zy:limp}.  Given $t \in (0,T_1)$ and $x \in \R$, let $t_\ep\to t$, and $x_\ep\to x$. Since $t\in(0,T)$,  $\tT_0=o_\ep(1)$ and 
 $\o T_1=T_1+o_\ep(1)$, we may assume that  $\tT_0< t_\ep<\o T_1$ for $\ep $ small enough. 
 By \eqref{pfxe} (recall \eqref{pf:yz} and \eqref{thetaepthetaprop}) we have that 
 \beqs \limsup_{\ep\to0} \o v_\e(t_\ep,x_\ep)= \limsup_{\ep\to 0} \sum_{i=1}^N u  \left(\displaystyle \frac{x_\e - \xs_i(t_\e - \tT_0)}{\ep} ; b_i \right).\eeqs 
 Note that by \eqref{uniformconvergencexbartox}, $x_\e - \xs_i(t_\e - \tT_0)\to x-x_i(t)$ as $\ep \to0$. Then 
 by  Lemma \ref{l:u:to:H} 
\begin{equation*} 
  \limsup_{ \e \to 0 }  \sum_{i=1}^N  u  \left(\displaystyle \frac{x_\e - \xs_i(t_\e - \tT_0)}{\ep} ; b_i \right) 
  \leq \sum_{i=1}^N(b_i H)^* (x - x_i(t)),
\end{equation*}  
which (recall \eqref{v*}) equals $v^*(t,x)$.
This proves \eqref{pf:zy:limp}.

The proof of \eqref{pf:zy} is almost complete; it only remains to treat the case in which no collision happens. This case is simpler, because no collision implies $\theta > 0$ on $[0, T]$. Then, by Lemma \ref{l:obx}, $\o\theta>\theta_\ep$ on $[0,T]$ for $\ep$ small enough and  \eqref{pf:zz} follows as in the previous case.

\subsubsection{Proof of \eqref{pf:zz} in $(0,T_1)$} 
Inequality \eqref{pf:zz} in $(0,T_1)$ follows from \eqref{pf:zy:supersol}, \eqref{pf:zy:limp}, \eqref{Ttildezero} and \eqref{pf:yu}.

\subsection{Resolving a simple collision}
\label{s:pf:15}

In this section we prove \eqref{pf:zz} at $T_1$ and on $(T_1, T_2)$  in  the case in which $\bx^0$ is such that the solution $\bx$ to \eqref{PN} has precisely 2 particles, say $x_k$ and $x_{k+1}$, colliding at time $T_1$ (simple collision). We divide two cases depending on the sign of $b_k$. Since $b_{k+1} = -b_k$, we call the case $b_k = 1$ a `$+-$' collision and the case $b_k = -1$ a `$-+$' collision. We mainly focus on the `$+-$' collision (Subsections \ref{subsubvtildel}, \ref{s:pf:1:3} and \ref{s:pf:1:4}), as this is the most challenging of the two (this would be different for the construction of subsolutions). We treat `$-+$' collisions in Subsection \ref{-+seclimsupproof}.

\subsubsection{Construction of $\hat v_\e^\ell$ in $[\o T_1 , \hat T_1]$ for a `$+-$' collision}\label{subsubvtildel}
  
Following  \cite{PatriziValdinoci16},  we construct $\hat v_\e^\ell$ which satisfy \eqref{pf:zx}. We mainly focus on $\hat v_\e^1$. Its construction relies again on a perturbed ODE system of the form \eqref{oPN}.

Let $l_0 > 0$ 
 be such that  
\begin{equation} \label{pf:zq}  
  \theta_i(t) \geq l_0
  \qquad \text{for } i \in \{1 ,\ldots, N-1\} \setminus \{k\}.
\end{equation} 
Recalling the $\e$-dependent particle positions $\o \bx$ from the construction of $\o v_\e$ and \eqref{otaudef}, we set
\beqs 
 \xs_i^\ep:=\xs_i(T_1 - \o \tau_\e),\quad 
\cs_i^\ep := \cs_i(T_1 - \o \tau_\e) = \dot \xs_i( T_1 - \o \tau_\e) \qquad \text{for } i=1,\ldots, N. 
\eeqs
Note from  \eqref{xbartetoxoft} that
\begin{equation} \label{pf:zr}
  | \xs_i^\ep - x_i(T_1) |=o_\e(1)\quad \text{for } i=1,\ldots, N. 
  \end{equation}
Then by \eqref{pf:zq}, for $\e$ small enough
\beq\label{pf:zl}  
\xs_{i+1}^\ep- \xs_{i}^\ep
\geq \frac{l_0}2 
\quad \text{for } i \neq k
\eeq
and by \eqref{thetabar=thetaepbis}
\beq\label{pf:ys}
\theta_\e
=
\o \theta (T_1 - \o \tau_\e)
\le \xs_{k+1}^\ep - \xs_k^\ep
=: \Theta_\ep 
= o_\e(1). 
\eeq
From the ODE in \eqref{oPN},  \eqref{thetabar=thetaepbis}  and \eqref{pf:zl} we infer that, for $i=1,\ldots, N$, 
 \beq \label{pf:zm}  
   |\cs_i^\ep| 
   \leq C + \begin{cases}
   \dfrac1{\theta_\ep}
   & \text{if } i \in \{ k, k+1 \} \\
   0 & \text{otherwise.}
\end{cases}    
 \eeq
\begin{rem} \label{r:Thetae}
Since $x_k$ and $x_{k+1}$ are the only colliding particles at $T_1$, we actually have in \eqref{pf:ys} that $\o \theta (T_1 - \o \tau_\e)= \xs_{k+1}^\ep - \xs_k^\ep$ for $\ep$ small enough.
Therefore,   $\theta_\ep=\Theta_\ep $ in \eqref{pf:ys}. 
However, in view of the generalization to the current arguments in Sections \ref{s:pf:2} and \ref{s:pf:3}, we will only use that 
$\theta_\ep\le \Theta_\e = o_\e(1)$ in this subsection. 
\end{rem}

Now, for parameters $L > 1$ and 
\begin{equation} \label{pf:wi}
  \hat \sigma = \hat\sigma_\e \geq \delta_\e 
  \quad \text{with} \quad   \hat \sigma = o_\e(1),
\end{equation}
let  $\hat \bx$ be the solution of the perturbed ODE system
 \beq\label{pf:zs}
\left\{ \begin{aligned}
   \dot{\xss}_i 
   &= \displaystyle\sum_{j \neq i} 
\displaystyle\frac{b_i b_j}{ \xss_i-\xss_j} - b_i \hat \sigma 
 &&\text{in }(0, \hsT_1) \text{ for } i = 1,\ldots, N \\
 \xss_i(0) 
 &= \xs_{i}^\ep - b_i \Theta_\ep 
 &&\text{if } i \neq k+1 \\
 \xss_i(0) 
 &= \xs_i^\ep - L b_i \Theta_\epsilon
 &&\text{if } i = k+1,
 \end{aligned} \right. 
\eeq
where $\hsT_1$ is the first collisions time of $\hat \bx$, and $\Theta_\e$ is as in \eqref{pf:ys}. Notice that \eqref{pf:zs} is equivalent to the ODE system in \eqref{oPN} with the parameter choice $\o \sigma = \hat \sigma$ and
\[
  \o x_i^0 = \begin{cases}
    \xs_{i}^\ep - \hat b_i \Theta_\ep 
    &\text{if } i \neq k+1 \\
    \xs_{i}^\ep - \hat b_i L \Theta_\ep 
    &\text{if } i = k+1.
  \end{cases}
\]
The parameter $L$ controls the asymmetry; see Figure \ref{fig:ve:oT}.

Before constructing $\hat v_\e^1$ from $\hat \bx$, we first prove several properties of $\hat \bx$. For $\e$ small enough, we have by \eqref{pf:zl} that the particles $\hat x_i$ are ordered, i.e.\ $\hat \bx \in \Omega^N$. Similar to \eqref{theta:defs}, we set
\beqs 
\hat {\theta}_i := \hat x_{i+1} - \hat x_i
\quad \text{and} \quad
\hat {\theta} :=\min_{1 \leq i \leq N-1} \hat {\theta}_i 
\qquad \text{on }  [0, \hsT_1). \eeqs

\begin{lem}\label{l:te}
There exist a universal $L>1$ and an $\e_0 > 0$ such that for all $\ep \in (0, \e_0)$ there holds
\beq\label{te} \hat \tau_\e 
:= \frac{L^2}6 \Theta_\e^2
< \hsT_1,
\eeq
where $\Theta_\e$ is as in \eqref{pf:ys},  
 and the solution 
 $\hat \bx (t)$ of \eqref{pf:zs} satisfies for $i = 1,\ldots,N$
\begin{align}  
  \label{l:te:00} \min_{t \in [0, \hat \tau_\e]} \hat \theta_i (t)
  &\geq \frac{l_0}4 
  \quad \text{if } i \neq k, \\   
  \label{l:te:1} \min_{t \in [0, \hat \tau_\e]} \hat \theta(t)
  &\ge\theta_\ep, \\
  \label{l:te:0} \max_{t \in [0, \hat \tau_\e]} \big| \dot {\hat x}_i(t) \big| 
  &\le \frac2{\theta_\ep},  \\
  \label{l:te:2} \max_{t \in [0, \hat \tau_\e]} \big| \xss_i(t) - x_i(T_1) \big| 
  &= o_\ep(1), \\
  \label{l:te:3} \hat x_k (\hat \tau_\e) 
  &\ge \xs_{k+1}^\ep. 
\end{align} 
  \end{lem}

\begin{proof}
The proof is similar in spirit to the proof of Lemma \ref{l:obx}. 
 Here, the situation is simpler because only two particles collide, but on the other hand more quantitative estimates need to be established. Therefore, we give an independent proof in full detail.

Take $L = 26$. We start by establishing several estimates on $\hat x_i$ and $\hat \theta_i$. At initial time, we observe from \eqref{pf:zl} and \eqref{pf:zs} that
\begin{equation} \label{pf:xh}
  \hat \theta_i(0) \geq \frac{l_0}3
  \quad \text{for all } i \neq k
\end{equation}
for $\e$ small enough. Hence, recalling \eqref{pf:ys},
\begin{equation} \label{pf:xg}
  \hat \theta (0) 
  = \hat \theta_k(0) 
  = (L+2) \Theta_\e .
\end{equation}

Next, let $\tau_1$ be the largest value in $[0, \hsT_1]$ such that
\begin{equation} \label{pf:zk}
  \hat \theta_i(t) \geq \frac{l_0}4
  \quad \text{for all } t \in [0, \tau_1) \text{ and all } i \neq k.
\end{equation}
At a later stage we will prove that $\tau_1 = \hsT_1$. For now,
the ODE in \eqref{pf:zs} implies that
\begin{equation} \label{pf:zh}
  \big| \dot{\hat x}_i \big| \leq C
  \quad \text{on } [0, \tau_1) \text{ for } i \neq k, k+1.
\end{equation}
This estimate does not hold for $i = k, k+1$. Instead, we show that $\dot{ \hat x}_{k+1} < 0 < \dot{ \hat x}_{k}$. With this aim
we obtain from the ODE in \eqref{pf:zs} and \eqref{pf:zk} that 
\begin{equation} \label{pf:zj}
  - \frac2{\hat \theta_k} - \frac{8N}{l_0} - 2 \hat \sigma
  \leq \dot{\hat \theta}_k
  \leq - \frac2{\hat \theta_k} + \frac{8N}{l_0} + 2 \hat \sigma
  \quad \text{on } [0, \tau_1).
\end{equation}
Since the right-hand side is increasing in $\hat \theta_k$ and negative at $t = 0$ for $\e$ small enough, we have $\dot{\hat \theta}_k < 0$ on $[0, \tau_1)$. In particular, for $\e$ small enough, $\dot{\hat \theta}_k \leq -1/\hat \theta_k$. Comparing this to the ODE $\dot \vartheta = -1/\vartheta$ with $\vartheta(0) = \hat \theta_k(0) = (L+2) \Theta_\e$, we obtain
\begin{equation} \label{pf:xa}
  \hat \theta_k(t)
  \leq \vartheta(t) = \sqrt{ (L+2)^2 \Theta_\e^2 - 2t }
  \qquad \text{for all } t \in [0, \tau_1).
\end{equation}
Then, using the ODE in \eqref{pf:zs} for $\hat x_k$ and applying similar estimates, we obtain for $\e$ small enough
\begin{equation} \label{pf:zg}
  \frac2{\hat \theta_k}
  \geq \dot{ \hat x}_k 
  \geq \frac1{2 \hat \theta_k} 
  \geq \frac1{2 (L+2) \Theta_\e}
  > 0
  \quad \text{on } [0, \tau_1).
\end{equation}
Similarly, we obtain for $\e$ small enough that
\begin{equation} \label{pf:yy} 
  - \frac2{\hat \theta_k}
  \leq \dot{ \hat x}_{k+1} 
  \leq - \frac1{2 (L+2) \Theta_\e}
  < 0
  \quad \text{on } [0, \tau_1).
\end{equation} 
Similar to the derivation of \eqref{pf:xa} we derive a lower bound on $\hat \theta_k$. The left-hand side in \eqref{pf:zj} is bounded from below by $-3/{\hat \theta_k}$ for $\e$ small enough. Comparing this with the ODE $\dot \vartheta = -3/\vartheta$ with $\vartheta(0) = \hat \theta_k(0)$, we obtain
\begin{equation} \label{pf:wz}
  \hat \theta_k(t)
  \geq \sqrt{ (L+2)^2 \Theta_\e^2 - 6t }
  \qquad \text{for all } 0 \leq t < \min \left\{ \tau_1, \frac{(L+2)^2}6 \Theta_\e^2 \right\}.
\end{equation}
The bounds on the velocities on $\hat x_i$ and \eqref{pf:xh} imply that
\begin{multline*} 
  \hat \theta_i(t)
  \geq \hat \theta_i(0) - \int_0^t \big| \dot{\hat x}_i(s) \big| + \big| \dot{\hat x}_{i+1} (s) \big| \, ds
  \geq \frac{l_0}3 - 2C \tau_1
  \\
   \text{for all } t \in [0, \tau_1) \text{ and all } i \notin \{k-1,k,k+1\},
\end{multline*}
\[
  \hat \theta_{k-1}(t)
  \geq \hat \theta_{k-1}(0) + \int_0^t \dot{\hat x}_k(s) - \big| \dot{\hat x}_{k-1} (s) \big| \, ds
  \geq \frac{l_0}3 - C \tau_1
  \quad \text{for all } t \in [0, \tau_1),
\]
and, similarly, $\hat \theta_{k+1}(t) \geq {l_0}/3 - C \tau_1$ for all $t \in [0, \tau_1)$. This completes the preliminary estimates on $\dot{\hat x}_i$ and $\hat \theta_i$. 
\medskip

Next we use the estimates on $\hat x_i$ and $\hat \theta_i$ to prove Lemma \ref{l:te}. We start with \eqref{te}. From the lower bounds on $\hat \theta_i$, we observe that \eqref{pf:zk} holds with $\tau_1 = \min\{ c, \hsT_1 \}$ for some $\e$-independent $c > 0$. Furthermore, from \eqref{pf:xa} we obtain for $\e$ small enough that $\hat \theta = \hat \theta_k$ on $[0,\tau_1)$, and that $\tau_1 = o_\e(1)$. Hence,
\[
  \tau_1 = \hsT_1 = o_\e(1).
\] 
Moreover, from \eqref{pf:wz} we obtain that
\[
  \hsT_1 
  \geq \frac{(L+2)^2}6 \Theta_\e^2
  > \hat \tau_\e,
\]
which proves \eqref{te}. Then, \eqref{l:te:00} is given by \eqref{pf:zk}, \eqref{l:te:1} follows from \eqref{pf:wz}, and \eqref{l:te:0} follows from \eqref{pf:zh}, \eqref{pf:zg} and \eqref{pf:yy}. 

Next  we prove \eqref{l:te:2}. Fixing $i$, we expand and estimate
\begin{equation*} 
  \max_{t \in [0, \hat \tau_\e]}\big| \xss_i(t) - x_i(T_1) \big|
  \leq \max_{t \in [0, \hat \tau_\e]} \big| \xss_i(t) - \xss_i(0) \big|
  + \big| \xss_i(0) - \o x_i^\e \big|
  + \big| \o x_i^\e - x_i(T_1) \big|.
\end{equation*}
From \eqref{pf:zs} and \eqref{pf:zr} we observe that the second and third term in the right-hand side are $o_\e(1)$. By \eqref{pf:zh}, the first term is also $o_\e(1)$ when $i \notin \{k, k+1\}$. If $i = \{k, k+1\}$, we use the monotonicity of $\dot{\hat x}_i$ to deduce that
\[
  \hat x_k(0) \leq \hat x_k(t) \leq \hat x_{k+1}(t) \leq \hat x_{k+1}(0)
  \quad \text{for all } t \in [0, \hat \tau_\e]. 
\]
Then, from \eqref{pf:xg} we obtain $| \xss_i(t) - \xss_i(0) | = o_\e(1)$. This proves \eqref{l:te:2}.  

Finally, we prove \eqref{l:te:3}. Using \eqref{pf:zg}, we compute
\[
  \hat x_k(\hat \tau_\e)
  = \hat x_k(0) + \int_0^{\hat \tau_\e} \dot{ \hat x}_k
  \geq \o x_k^\e - \Theta_\e + \frac{\hat \tau_\e}{2 (L+2) \Theta_\e}
  = \o x_{k+1}^\e - 2 \Theta_\e + \frac{L^2}{12 (L+2)} \Theta_\e.
\]
Substituting $L = 26$, we obtain \eqref{l:te:3}.
\end{proof}

With Lemma \ref{l:te} in hand, we return to proving \eqref{pf:zx} for $\ell = 1$. Similar to \eqref{pf:yz}, we set
\beq\label{pf:zf}\begin{split}
\hat v_\ep^1(t,x)
&:=\sum_{i=1}^N ( u - \ep \hat c_i(t - \o T_1)\psi ) \left(\displaystyle \frac{x- \hat x_i(t - \o T_1)}{\ep} ; b_i \right) 
   + \frac\ep\alpha (\hat \sigma - \delta_\e) \end{split} 
\eeq
for $t \in [ \o T_1, \o T_1 + \hat \tau_\e]$ and $x \in \R$, where $\hat c_i := \dot{ \hat x}_i$ and $ \hat \sigma $ is chosen below in \eqref{pf:wj}.

We start with proving \eqref{pf:zx:IC}. Note that 
\begin{multline*} 
   \big( \hat v_\ep^1 - \vs_\ep \big) (\o T_1, x)
= \frac\ep\alpha (\hat \sigma - \o \sigma) \\
+\sum_{i=1}^N \left[ ( u - \ep \hat c_i(0)\psi ) \left(\displaystyle \frac{x-  \xs_i^\e + L_i b_i \Theta_\e}{\ep} ; b_i \right)  - ( u - \ep \cs_i^\e \psi ) \left(\displaystyle \frac{x-\xs_i^\e}{\ep} ; b_i \right) \right],
\end{multline*} 
where $L_i := 1$ if $i \neq k+1$ and $L_{k+1} := L$.
For each $i$ we apply Lemma \ref{l:u:epsi:new} (with $\vartheta_\e = \theta_\e$) to the summand. From $\Theta_\e \geq \theta_\e$ (recall \eqref{pf:ys}) and the bounds on the velocities in \eqref{pf:zm} and \eqref{l:te:0} it follows that the conditions of Lemma \ref{l:u:epsi:new} are met. Applying Lemma \ref{l:u:epsi:new} we obtain
\begin{equation*} 
  (\hat v_\e^1 - \o v_\e) (\o T_1, x)
  \geq  \frac\ep\alpha ( \hat \sigma - \o \sigma ) - C_0 \frac{\e^2}{\theta_\e^2 }
  \quad \text{for all } x \in \R
\end{equation*}
for some constant $C_0 > 0$ independent of $\e$. Hence, taking
\begin{equation} \label{pf:wj}
   \hat \sigma := \o \sigma + \alpha C_0 \frac{\e}{\theta_\e^2} = o_\e(1),
\end{equation}
\eqref{pf:zx:IC} follows and $\hat \sigma$ satisfies the requirements in \eqref{pf:wi} (recall \eqref{thetaepthetaprop}). 

Next we prove \eqref{pf:zx:supersol}. Lemma \ref{l:vebar} and \eqref{l:te:1} 
 imply that the function $\hat v^1_\e$ is a supersolution of \eqref{HJe:formal} on $[\o T_1, \o T_1 + \hat \tau_\e]$. Then, the comparison principle in Proposition \ref{comparisonuep}  yields \eqref{pf:zx:supersol} with
\begin{equation*} 
  \hat T_1 := \o T_1 + \hat \tau_\e =T_1+o_\ep(1).
\end{equation*}

In preparation for proving the bound \eqref{pf:wy}\footnote{\eqref{pf:wy} is stated for $k=1$; the minor changes for general $k$ are obvious} on $\hat v_\e^3$, we derive a sufficient upper bound for $\hat v_\e^1(\hat T_1, x)$. More precisely, we show that for all $\e > 0$ small enough and all $x \leq (\o x_k^\e + \o x_{k+1}^\e)/2$
\beq\label{pf:ze} 
\hat v^1_\ep(\hat T_1,x)
\leq \sum_{i \in S_1} u\left(\displaystyle \frac{x - \hat x_i(\hat \tau_\e)}{\ep} ; b_i\right) + \hat \rho_\e
\eeq
for some constant $\hat \rho_\e = o_\e(1)$, where 
\begin{equation*} 
  S_1 := S_{T_1} = \{1,\ldots,N\} \setminus \{k,k+1\}.
\end{equation*}
To prove \eqref{pf:ze} we write the left-hand side as in \eqref{pf:zf}. Then, it follows from  $\hat \sigma = o_\e(1)$ and the bound on $|{\hat c}_i|$ in \eqref{l:te:0} that it is sufficient to show that
\begin{equation} \label{pf:zd}
  u\left(\displaystyle \frac{x - \hat x_k(\hat \tau_\e)}{\ep} ; +1\right) + u\left(\displaystyle \frac{x - \hat x_{k+1}(\hat \tau_\e)}{\ep} ; -1 \right)
  \leq \hat \rho_\e
\end{equation}
for some $\hat \rho_\e = o_\e(1)$. 
Note that the left-hand side equals
\[
  u\left(\displaystyle \frac{x - \hat x_k(\hat \tau_\e)}{\ep} \right) + u\left(\displaystyle \frac{\hat x_{k+1}(\hat \tau_\e) - x}{\ep} \right) - 1.
\]
We bound the second term from above simply by $1$. For the first term, we use that $u$ is increasing, $\hat x_k(\hat \tau_\e) \geq \o x_{k+1}^\e$ and $x \leq (\o x_k^\e + \o x_{k+1}^\e)/2$ to get 
\[
  u\left(\displaystyle \frac{x - \hat x_k(\hat \tau_\e)}{\ep} \right)   
  \le u\left(\displaystyle \frac{\o x_k^\e - \o x_{k+1}^\e}{2 \ep} \right)
  = u\left(\displaystyle - \frac{\Theta_\e}{2 \ep} \right)
  \leq u\left(\displaystyle - \frac{\theta_\e}{2 \ep} \right),
\]
which by Lemma \ref{l:u} is bounded by $3 \e / (\alpha \pi \theta_\e) = o_\e(1)$ for $\e$ small enough. Hence, \eqref{pf:zd} and consequently the claim follow.
\medskip 

Next we construct $\hat v_\e^2$. The construction is similar to that of $\hat v_\e^1$. To distinguish it, we change the notation in the construction of $\hat v_\e^1$ to 
$$ 
  \hat \bx^1 := \hat \bx
  , \quad
  \hat \rho_\e^1 := \hat \rho_\e.
$$
We construct $\hat \bx^2$ analogously to $\hat \bx^1$ as the solution to \eqref{pf:zs}, with the only difference that the constant $L$ in the initial condition is swapped from particle $k+1$ to $k$, i.e.\ we set instead
\[
  \xss_k^2(0) 
 := \xs_k^\ep - L b_k \Theta_\ep, \qquad
 \xss_{k+1}^2(0) 
 := \xs_{k+1}^\ep - b_{k+1} \Theta_\epsilon.
\]
Then, by taking $\e$ even smaller if needed, Lemma \ref{l:te} still holds with the same value for $\hat \tau_\e$ and $L$, and only \eqref{l:te:3} changes into
\begin{equation*}
  \hat x_{k+1}^2 (\hat \tau_\e) \le \xs_k^\ep.
\end{equation*}
Analogously to the case of $\hat v_\e^1$, we obtain that $\hat v_\e^2$ defined as in \eqref{pf:zf} with respect to $\hat \bx^2$ satisfies \eqref{pf:zx}, and that instead of \eqref{pf:ze} we obtain that there exists a constant $\hat \rho_\e^2 = o_\e(1)$ such that for all $\e > 0$ small enough and all $x \geq (\o x_k^\e + \o x_{k+1}^\e)/2$
\beqs 
\hat v_\ep^2(\hat T_1,x)
\leq \sum_{i \in S_1} u\left(\displaystyle \frac{x - \hat x_i^2(\hat \tau_\e)}{\ep} ; b_i\right) + \hat \rho_\e^2.
\eeqs

Combining this estimate with the one on $\hat v_\e^1$ in \eqref{pf:ze} and with \eqref{pf:zx}, we obtain that
$
  \hat v_\e^3 = \min \{ \hat v_\e^1, \hat v_\e^2 \}
$
satisfies 
\begin{equation} \label{pf:zc}
  v(\hat T_1,x)
   \leq \hat v_\e^3 (\hat T_1,x)
   \leq \sum_{i \in S_1} u\left(\displaystyle \frac{x - \hat x_i^\e}{\ep} ; b_i\right) + \hat \rho_\e^3
   =: \tilde v_\e (\hat T_1, x)
   \qquad \text{for all } x \in \R,
\end{equation}   
where 
\begin{equation} \label{pf:xd}
\hat \rho_\e^3 := \max \{ \hat \rho_\e^1, \hat \rho_\e^2 \} = o_\e(1),
\qquad 
  \hat x_i^\e :=
  \begin{cases}
  \min \{ \hat x_i^1, \hat x_i^2 \}(\hat \tau_\e)
  &\text{if } b_i = 1 \\
  \max \{ \hat x_i^1, \hat x_i^2 \}(\hat \tau_\e)
  &\text{if } b_i = -1
  \end{cases} 
\end{equation}
for all $i \in S_1$.     

\subsubsection{Supersolutions for the `$+-$' collision beyond $\hat T_1$}
\label{s:pf:1:3} 
In preparation for proving \eqref{pf:zz} at $t = T_1$, we extend our construction of supersolutions beyond $\hat T_1$ to some time point which is larger than $T_1$ uniformly in $\e$. By replacing in Subsection \ref{subsubvtildeep} the parameters $\rho_\e$ and $\bx_\e^0$ by respectively $\hat \rho_\e^3$ and $\{ \hat x_i^\e \}_{i \in S_1}$ defined in \eqref{pf:xd}, we observe from \eqref{l:te:2} and \eqref{pf:zc} that the construction in Subsection \ref{subsubvtildeep} yields a $\tT_1 > \hat T_1$ and a supersolution $\tilde v_\e$ such that \eqref{pf:zw} holds on  the time interval $[\hat T_1, \tT_1]$ for $\e$ small enough. Then, the construction in Subsection \ref{subsubvepbar} yields a supersolution $\o v_\ep$ on $[\tilde T_1, \o T_2]$ with $\o T_2 = T_2 + o_\e(1)$, which is larger than $T_1$ uniformly in $\e$ for $\e$ small enough. This proves \eqref{pf:zz} on $(T_1, T_2)$ for a `$+-$' collision at $T_1$.

\subsubsection{Proof of \eqref{pf:zz}  at $t = T_1$ for a `$+-$' collision} \label{s:pf:1:4} 

Here we prove \eqref{pf:zz} at $T_1$ under the assumption that $b_k=1$. Precisely, we show that for all $x \in \R$, $x_\e \to x$ and $t_\e \to T_1$ as $\ep\to0$ that
\begin{equation} \label{pf:xt}
  \limsup_{\e \to 0} v_\e(t_\e, x_\e) 
  \leq v^*(T_1, x)
  = \sum_{i = 1}^N (b_i H)^* (x - x_i(T_1)).
\end{equation}
The equality in \eqref{pf:xt} follows from \eqref{v*} by noting that
\[
  \chi(T_1, x) = H^* (x - x_k(T_1)) + (-H)^* (x - x_{k+1}(T_1)). 
\]

We make use of the supersolutions $\o v_\e |_{[\tT_0, \o T_1)}, \hat v_\e^1, \tilde v^1_\e$ and $\o v_\e |_{[\tT_1, \o T_2)}$ which all bound $v_\e$ from above.However, these bounds hold on the different time intervals separated by $\o T_1, \hat T_1, \tT_1$ (see Figure \ref{fig:timeline}). Since $\o T_1, \hat T_1, \tT_1$ can all be expressed as $T_1 + o_\e(1)$, it depends on $t_\e$ which of the four supersolutions bounds $v_\e$ from above. We therefore split four cases depending on $t_\e$.

If, along a subsequence (not relabelled), $t_\e \in [\tT_0, \o T_1)$, then by \eqref{pf:zy:supersol} (recalling \eqref{pf:yz}) 
\begin{equation} \label{pf:yq}   
  v_\e(t_\e, x_\e)
  \leq \o v_\e(t_\e, x_\e)
  \leq \sum_{i=1}^N (u - \e \o c_i(t_\e - \tT_0) \psi ) \left(\displaystyle \frac{x_\e - \xs_i(t_\e - \tT_0)}{\ep} ; b_i\right) + o_\e(1).
\end{equation}
By \eqref{pfxe} (recalling $T_1 - \o \tau_\e = \o T_1 - \tT_0$ and \eqref{thetaepthetaprop}) the terms related to $\psi$ can be absorbed in the remainder $o_\e(1)$. For the term related to $u$, we have by \eqref{xbartetoxoft}  that $\xs_i(t_\e - \tT_0) = x_i(T_1) + o_\e(1)$. Then, using Lemma \ref{l:u:to:H}, \eqref{pf:xt} follows for the extracted subsequence of $\e$.

If, along a subsequence (not relabeled), $t_\e \in [\o T_1, \hat T_1)$, then by \eqref{pf:zx:supersol},  \eqref{pf:xt} follows for this subsequence by replacing $\o v_\e$ in \eqref{pf:yq} by $\hat v_\e^1$ and by proceeding in a similar manner. Indeed, \eqref{l:te:0} provides a sufficient bound on $\hat c_i$ and \eqref{l:te:2} shows that $\hat x_i(t_\e - \o T_1) = x_i(T_1) + o_\e(1)$. 

If, along a subsequence (not relabeled), $t_\e \in [\hat T_1, \tT_1)$, then (recalling \eqref{pf:za} and $S_1 = S_{T_1}$)
\begin{equation*} 
  v_\e(t_\e, x_\e)
  \leq \tv_\e(t_\e, x_\e)
  \leq \sum_{i \in S_1} u\left(\displaystyle \frac{x_\e - \tx_i(t_\e - \hat T_1)}{\ep} ; b_i\right) + o_\e(1).
\end{equation*} 
Here, $\tilde x_i$ is defined similar to \eqref{pf:yv}, but with $\bx_\e^0$ replaced by $\{ \hat x_i^\e \}_{i \in S_1}$ (recall \eqref{pf:xd}). Again, a similar estimate as \eqref{pf:yw} holds, from which we infer $\tx_i(t_\e - \hat T_1) = x_i(T_1) + o_\e(1)$.  Then, using Lemma \ref{l:u:to:H}, we obtain
\[
  \limsup_{\e \to 0} v_\e(t_\e, x_\e) 
  \leq \sum_{i \in S_1} (b_i H)^* (x - x_i(T_1))
  \leq v^*(T_1, x).
\]
The case $t_\e \in [\tT_1, \o T_2)$ can be treated similarly as the case $t_\e \in [\tT_0, \o T_1)$; we omit the details. This completes the proof of \eqref{pf:xt}. 

\subsubsection{The case of a `$-+$' collision}
\label{-+seclimsupproof}
Here, we construct $\hat v_\e^\ell$ and prove \eqref{pf:zz} on $[T_1, T_2)$ in the remaining case
$b_k = -1$. 

The construction of $\hat v^\ell_\e$ in Subsection \ref{subsubvtildel} can be skipped completely; see Figure \ref{fig:ve:oT:inversed} for a sketch. We simply set $\hat T_1 = \o T_1$ and show that 
\begin{equation} \label{pf:wx} 
  \vs_\ep(\o T_1, x)
=\sum_{i=1}^N ( u - \ep \cs_i^\e \psi ) \left(\displaystyle \frac{x-\xs_i^\e}{\ep} ; b_i \right)  + \frac\ep\alpha (\o \sigma-\delta_\e)
\end{equation}
can be bounded from above similarly to the bound on $\hat v_\e$ in \eqref{pf:zc}. Recalling \eqref{pf:zm}, the terms in $\vs_\ep(\o T_1, x)$ related to $\psi$ and the constant are $o_\e(1)$ uniformly in $x$. For the colliding particles, we observe from $\o x_{k+1}^\e - \o x_k^\e \geq \theta_\e$ that Lemma \ref{l:u:epsi:dipole} applies with $c = c' = 0$ and $\vartheta_\e = \theta_\e$. 
This application yields 
\begin{equation*} 
  u\left(\displaystyle \frac{x - \o x_k^\e}{\ep} ; b_k \right) + u\left(\displaystyle \frac{x - \o x_{k+1}^\e}{\ep} ; b_{k+1} \right)
  \leq C \frac{\e^2}{\theta_\e^2}.
\end{equation*}
Substituting this into the expression of $\vs_\ep(\o T_1, x)$, we obtain 
\begin{equation} \label{pf:xs}
v_\ep(\o T_1, x)\le  \o v_\e (\o T_1, x)
\leq \sum_{i \in S_1} u\left(\displaystyle \frac{x - \o x_i^\e}{\ep} ; b_i\right) 
   + \tilde \rho_\e
\end{equation}
for some constant $\tilde \rho_\e = o_\e(1)$, which is qualitatively similar to the bound in \eqref{pf:zc}, as desired.

Using \eqref{pf:xs}, the construction in Subsection \ref{s:pf:1:3} of $\tv_\e$ and $\o v_\ep$ in 
$[\hat T_1, \tilde T_1]$ and $[\tilde T_1, \o T_2]$ applies with obvious modifications. This proves \eqref{pf:zz} on $(T_1, T_2)$.

It remains to prove \eqref{pf:zz} at $t = T_1$. The difference with the setting of the `$+-$' collision in Section \ref{s:pf:1:4} is that the value of $v^*(T_1, x)$ is different if and only if $x = x_k(T_1)$ (see \eqref{v*}). Consequently, instead of \eqref{pf:xt}, we have to prove 
\begin{equation*} 
  \limsup_{\e \to 0} v_\e(t_\e, x_\e) 
  \leq \sum_{i \in S_1} (b_i H)^* (x - x_i(T_1))
\end{equation*}
for all $x \in \R$, $x_\e \to x$ and $t_\e \to T_1$ as $\ep\to0$. The proof in Section \ref{s:pf:1:4} applies with mainly obvious modifications; the only nontrivial modification is that we require the following claim: if $t_\e \in [\tT_0, \o T_1)$ along a subsequence, then 
\begin{equation*}  
  \limsup_{\e \to 0} u\left(\displaystyle \frac{x_\e - \xs_k(t_\e - \tT_0)}{\ep} ; -1 \right)
  + u\left(\displaystyle \frac{x_\e - \xs_{k+1}(t_\e - \tT_0)}{\ep} ; +1 \right)
  \leq 0.
\end{equation*}
This claim follows from Lemma \ref{l:u:epsi:dipole} with $c = c' = 0$ and $\vartheta_\e = \theta_\e \leq \xs_{k+1}(t_\e - \tT_0) - \xs_k(t_\e - \tT_0)$ (recall \eqref{pf:yu}, \eqref{thetabar>=thetaep} and \eqref{thetaepthetaprop}).

\subsection{Multiple-particle collisions at one point}
\label{s:pf:2}

In this section we treat the case in which $\bx^0$ is such that at  collision time $T_1$, there is exactly one point $y \in \R$ at which the collision takes place.
  We denote  the set of indices of particles $x_i$ colliding at $(T_1, y)$ as
\[
  I := \{ k, k+1, \ldots, K \}.
\]
By Proposition \ref{p:vMPP} the orientations of the colliding particles are alternating, i.e.\
\begin{equation} \label{pf:ym} 
  b_i b_{i+1} = -1
  \quad \text{for all } k \leq i \leq K-1.
\end{equation}

We start from $\tv_\e$ and $\o v_\e$ as constructed in Section \ref{s:pf:1}. Then, we construct $\hat v_\e^\ell$ and the subsequent $\tv_\e$ and $\o v_\e$ in Subsection \ref{s:pf:2:1} by modifying the construction in Section \ref{s:pf:15}. We prove \eqref{pf:zz} at $T_1$ in Subsection \ref{proofatT1secgeneral}.

\subsubsection{Construction of $\hat v_\e^\ell$, $\tv_\e$ and $\o v_\e$ on $[\o T_1, \o T_2]$} 
\label{s:pf:2:1}

We recall that
\begin{equation} \label{pf:yi}
   \vs_\ep(\o T_1, x)
=\sum_{i=1}^N ( u - \ep\cs_i^\e \psi ) \left(\displaystyle  \frac{x-\xs_i^\e}{\ep}; b_i \right) + \frac \ep\alpha (\o \sigma - \delta_\e).
\end{equation} 
While \eqref{pf:zr} still holds as stated in the current setting, most of the properties of $\o c_i^\e$ and $\o x_i^\e$ hold with minor changes. For instance, instead of \eqref{pf:zq}, we have
\begin{equation} \label{pf:yn} 
  \theta_i(T_1) \geq l_0
  \qquad \text{for all } i \leq k - 1 \text{ and all } i \geq K
\end{equation}
for some fixed $l_0 > 0$. Then, we obtain instead of \eqref{pf:zl} and \eqref{pf:ys} that
\beqs 
\xs_{i+1}^\ep- \xs_{i}^\ep
\geq \frac{l_0}2 
\quad \text{for all } i \leq k - 1 \text{ and all } i \geq K
\eeqs
and
\beq\label{pf:yk}
\theta_\e
=
\o \theta (T_1 - \o \tau_\e)
\leq \xs_K^\ep - \xs_k^\ep
=: \Theta_\ep = o_\e(1),
\eeq
where the last equality follows from \eqref{pf:zr} and $x_k(T_1) = x_K(T_1)$. 

Finally, a similar estimate on $|\cs_i^\ep|$ as \eqref{pf:zm} holds. However, its derivation becomes more involved since $K-k$ many particles are $o_\e(1)$ close to each particle $x_i$ with $i \in I$. 
By \eqref{pf:ym} we obtain for the contribution of the particles with $j \in I$ to $\cs_i^\ep = \dot{ \o x}_i (\o T_1 - \ttau_\e)$ with $i \in I$ that
\[
   \sum_{ j = k }^{i-1} 
 \frac{b_i b_j }{ \xs_i-\xs_j}
  =  \sum_{ j = k }^{i-1} 
 \frac{ (-1)^{i-j} }{ \xs_i-\xs_j}
\in [-\theta_\e^{-1}, 0]
\] 
and 
\[
  \displaystyle\sum_{ j = i+1 }^K 
\displaystyle\frac{b_i b_j }{ \xs_i-\xs_j}
  = \displaystyle\sum_{ j = i+1 }^K 
\displaystyle\frac{ (-1)^{j-i+1} }{ \xs_j - \xs_i}
\in [0, \theta_\e^{-1}].
\] 
Then, from a similar derivation that led to \eqref{pf:zm}, we obtain  for $i=1,\ldots, N$ that
 \beqs 
   |\cs_i^\ep| 
   \leq C + \begin{cases}
   \dfrac1{\theta_\ep}
   & \text{if } k \leq i \leq K \\
   0 & \text{otherwise.}
\end{cases}    
 \eeqs

We recall that in Section \ref{s:pf:1} we split two cases depending on the sign of $b_k$. Here, we split an additional two cases depending on whether 
$$
\# I = K-k+1
$$ is even or odd, yielding a total of four cases. For each of these four cases Figure \ref{fig:mult-part-coll} illustrates a schematic of the easiest nontrivial situation. In view of Figure \ref{fig:mult-part-coll}, the idea is to use Lemma \ref{l:u:epsi:dipole} to remove the downward bumps in the graph of $\o v_\e$, which results in a graph which corresponds to either $0$, $1$ or $2$ particles.

\begin{figure}[ht]
\centering
\begin{tikzpicture}[scale=1.8, >= latex]
\draw[->] (0,-1.25) -- (0,.25);
\draw[->] (-.25,0) -- (5.25,0) node[right] {$x$};
\draw (0,-1) node[left]{$-1$} --++ (5.25,0);
\draw (-.2,-.5) node[left]{Case 1};

\showfig{
\draw[domain=-.25:1.5, smooth, thick, red] plot (\x,
  { -.55*( atan( 10*(\x - 1) )/90 + 1 ) 
    + .55*( atan( 10*(\x - 2) )/90 + 1 ) 
    - .55*( atan( 10*(\x - 2.75) )/90 + 1 )
    + .55*( atan( 10*(\x - 4) )/90 + 1 )
    + .04 });
\draw[domain=1.5:2.375, smooth, thick, red] plot (\x,
  { -.55*( atan( 10*(\x - 1) )/90 + 1 ) 
    + .55*( atan( 10*(\x - 2) )/90 + 1 ) 
    - .55*( atan( 10*(\x - 2.75) )/90 + 1 )
    + .55*( atan( 10*(\x - 4) )/90 + 1 )
    + .04 });
\draw[domain=2.375:3.375, smooth, thick, red] plot (\x,
  { -.55*( atan( 10*(\x - 1) )/90 + 1 ) 
    + .55*( atan( 10*(\x - 2) )/90 + 1 ) 
    - .55*( atan( 10*(\x - 2.75) )/90 + 1 )
    + .55*( atan( 10*(\x - 4) )/90 + 1 )
    + .04 });
\draw[domain=3.375:5.25, smooth, thick, red] plot (\x,
  { -.55*( atan( 10*(\x - 1) )/90 + 1 ) 
    + .55*( atan( 10*(\x - 2) )/90 + 1 ) 
    - .55*( atan( 10*(\x - 2.75) )/90 + 1 )
    + .55*( atan( 10*(\x - 4) )/90 + 1 )
    + .04 });            
\draw[thick, blue] (-.25,.06) -- (5.25,.06);

\draw[<->] (2,-.5) --++ (.75,0) node[midway, above]{$\geq \theta_\e$};
\draw[<->] (1,-.7) --++ (3,0) node[midway, below]{$\Theta_\e$};

\draw[red] (1,0) --++ (0,-.5);
\draw[red] (2,0) --++ (0,-.5);
\draw[red] (2.75,0) --++ (0,-.5);
\draw[red] (4,0) --++ (0,-.5);
\draw[red] (1,.06) node[above] {$\o x_k^\e$};
\draw[red] (2,.06) node[above] {$\o x_{k+1}^\e$};
\draw[red] (2.75,.06) node[above] {$\o x_{k+2}^\e$};
\draw[red] (4,.06) node[above] {$\o x_K^\e$};
\draw[red] (4, -.5) node[right] {$\o v_\e$};
\draw[blue] (3.375, .06) node[above]{$v_\e^1$};

\draw[<-] (5,.01) --++ (0,-.3) node[below] {$o_\e(1)$}; 
\draw[<-] (5,.08) --++ (0,.3);

\begin{scope}[shift={(0,-2.75)},scale=1] 
	\draw[->] (0,-.25) -- (0,1.25);
	\draw[->] (-.25,0) -- (5.25,0) node[right] {$x$};
	\draw (0,1) node[left]{$1$} --++ (5.25,0);
	\draw (-.2,.5) node[left]{Case 2};

	\draw[domain=-.25:2.375, smooth, thick, red] plot (\x,
	  { .55*( atan( 10*(\x - 2) )/90 + 1 ) 
	    - .55*( atan( 10*(\x - 2.75) )/90 + 1 )
	    + .55*( atan( 10*(\x - 4) )/90 + 1 )
	    + .00 });
	\draw[domain=2.375:3.375, smooth, thick, red] plot (\x,
	  { .55*( atan( 10*(\x - 2) )/90 + 1 ) 
	    - .55*( atan( 10*(\x - 2.75) )/90 + 1 )
	    + .55*( atan( 10*(\x - 4) )/90 + 1 )
	    + .00 });
	\draw[domain=3.375:5.25, smooth, thick, red] plot (\x,
	  { .55*( atan( 10*(\x - 2) )/90 + 1 ) 
	    - .55*( atan( 10*(\x - 2.75) )/90 + 1 )
	    + .55*( atan( 10*(\x - 4) )/90 + 1 )
	    + .00 });            
	\draw[domain=-.25:1.95, smooth, thick, blue] plot (\x,
	  { .55*( atan( 10*(\x - 1.95) )/90 + 1 ) 
	    + .05 }); 
	\draw[domain=1.95:5.25, smooth, thick, blue] plot (\x,
	  { .55*( atan( 10*(\x - 1.95) )/90 + 1 ) 
	    + .05 });    
	
	
	\draw[red] (2,0) node[below] {$\o x_k^\e$} --++ (0,.5);
	\draw[red] (2.75,0) node[below] {$\o x_{k+1}^\e$} --++ (0,.5);
	\draw[red] (4,0) node[below] {$\o x_K^\e$} --++ (0,.55) node[right] {$\o v_\e$};
	\draw[blue] (1.9, .6) node[left]{$ v_\e^1$};
	
\end{scope}

\begin{scope}[shift={(0,-3.75)},scale=1] 
	\draw[->] (0,-1.25) -- (0,.25);
	\draw[->] (-.25,0) -- (5.25,0) node[right] {$x$};
	\draw (0,-1) node[left]{$-1$} --++ (5.25,0);
	\draw (-.2,-.5) node[left]{Case 3};
	
	\draw[domain=-.25:1.5, smooth, thick, red] plot (\x,
	  { -.55*( atan( 10*(\x - 1) )/90 + 1 ) 
	    + .55*( atan( 10*(\x - 2) )/90 + 1 ) 
	    - .55*( atan( 10*(\x - 2.75) )/90 + 1 )
	    + .1 });
	\draw[domain=1.5:2.375, smooth, thick, red] plot (\x,
	  { -.55*( atan( 10*(\x - 1) )/90 + 1 ) 
	    + .55*( atan( 10*(\x - 2) )/90 + 1 ) 
	    - .55*( atan( 10*(\x - 2.75) )/90 + 1 )
	    + .1 });
	\draw[domain=2.375:5.25, smooth, thick, red] plot (\x,
	  { -.55*( atan( 10*(\x - 1) )/90 + 1 ) 
	    + .55*( atan( 10*(\x - 2) )/90 + 1 ) 
	    - .55*( atan( 10*(\x - 2.75) )/90 + 1 )
	    + .1 });     
	\draw[domain=-.25:2.8, smooth, thick, blue] plot (\x,
	  { - .55*( atan( 10*(\x - 2.8) )/90 + 1 )
	    + .14 });
	\draw[domain=2.8:5.25, smooth, thick, blue] plot (\x,
	  { - .55*( atan( 10*(\x - 2.8) )/90 + 1 )
	    + .14 });    
	
	
	\draw[red] (1,0) --++ (0,-.45) node[left] {$\o v_\e$};
	\draw[red] (2,0) --++ (0,-.45);
	\draw[red] (2.75,0) --++ (0,-.45);
	\draw[red] (1,.06) node[above] {$\o x_k^\e$};
	\draw[red] (2,.06) node[above] {$\o x_{k+1}^\e$};
	\draw[red] (2.75,.06) node[above] {$\o x_{K}^\e$};
	\draw[blue] (2.85, -.4) node[right]{$ v_\e^1$};
\end{scope}

\begin{scope}[shift={(0,-6.5)},scale=1] 
	\draw[->] (0,-.25) -- (0,1.25);
	\draw[->] (-.25,0) -- (5.25,0) node[right] {$x$};
	\draw (0,1) node[left]{$1$} --++ (5.25,0);
	\draw (-.2,.5) node[left]{Case 4};
	
	\draw[domain=-.25:1.5, smooth, thick, red] plot (\x,
	  { .55*( atan( 10*(\x - 1) )/90 + 1 ) 
	    - .55*( atan( 10*(\x - 2) )/90 + 1 ) 
	    + .55*( atan( 10*(\x - 2.75) )/90 + 1 )
	    - .55*( atan( 10*(\x - 4) )/90 + 1 )
	    + .02 });
	\draw[domain=1.5:2.375, smooth, thick, red] plot (\x,
	  { .55*( atan( 10*(\x - 1) )/90 + 1 ) 
	    - .55*( atan( 10*(\x - 2) )/90 + 1 ) 
	    + .55*( atan( 10*(\x - 2.75) )/90 + 1 )
	    - .55*( atan( 10*(\x - 4) )/90 + 1 )
	    + .02 });
	\draw[domain=2.375:3.375, smooth, thick, red] plot (\x,
	  { .55*( atan( 10*(\x - 1) )/90 + 1 ) 
	    - .55*( atan( 10*(\x - 2) )/90 + 1 ) 
	    + .55*( atan( 10*(\x - 2.75) )/90 + 1 )
	    - .55*( atan( 10*(\x - 4) )/90 + 1 )
	    + .02 });
	\draw[domain=3.375:5.25, smooth, thick, red] plot (\x,
	  { .55*( atan( 10*(\x - 1) )/90 + 1 ) 
	    - .55*( atan( 10*(\x - 2) )/90 + 1 ) 
	    + .55*( atan( 10*(\x - 2.75) )/90 + 1 )
	    - .55*( atan( 10*(\x - 4) )/90 + 1 )
	    + .02 }); 
	\draw[domain=-.25:2.5, smooth, thick, blue] plot (\x,
	  { .55*( atan( 10*(\x - .95) )/90 + 1 ) 
	    - .55*( atan( 10*(\x - 4.05) )/90 + 1 )
	    + .04 });
	\draw[domain=2.5:5.25, smooth, thick, blue] plot (\x,
	  { .55*( atan( 10*(\x - .95) )/90 + 1 ) 
	    - .55*( atan( 10*(\x - 4.05) )/90 + 1 )
	    + .04 });                      
	
	\draw[red] (1,0) node[below] {$\o x_k^\e$} --++ (0,.5);
	\draw[red] (2,0) node[below] {$\o x_{k+1}^\e$} --++ (0,.55);
	\draw[red] (2.75,0) node[below] {$\o x_{k+2}^\e$} --++ (0,.55);
	\draw[red] (4,0) node[below] {$\o x_K^\e$} --++ (0,.55) node[left] {$\o v_\e$};
	\draw[blue] (4.1, .6) node[right]{$\hat v_\e^0$};
\end{scope}
}
\end{tikzpicture} \\
\caption{Simple, typical examples for multiple particle collisions in Cases 1--4.}
\label{fig:mult-part-coll} 
\end{figure}
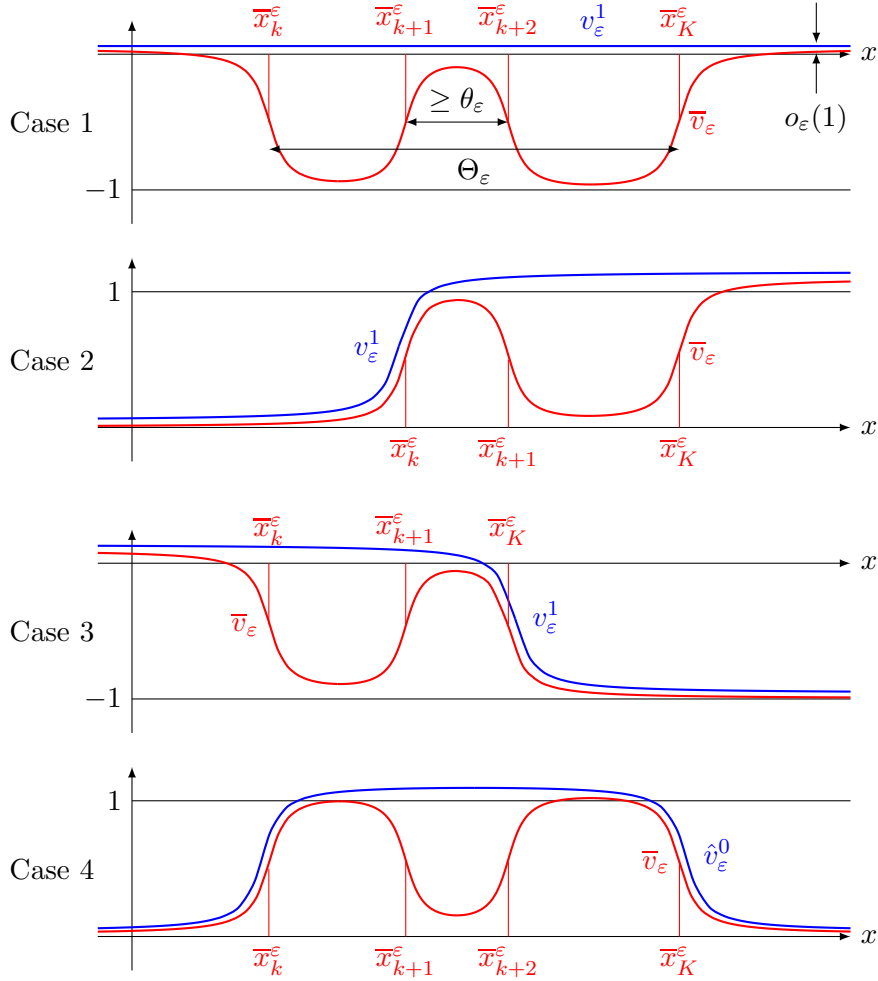

\paragraph{Case 1: $\# I$ is even and $b_k = -1$.} This case follows from a minor modification to the proof in Subsection \ref{-+seclimsupproof}, in which we skipped the construction of $\hat v_\e^\ell$, put $\hat T_1 := \o T_1$, and constructed an upper bound of the type \eqref{pf:xs}. In the current setting, with 
\[
  S_1 = \{1, \ldots, k-1\} \cup \{K+1, \ldots, N\}, 
\]
the only modification to the proof of \eqref{pf:xs} is that we apply Lemma \ref{l:u:epsi:dipole} with $\vartheta_\e = \theta_\e$ to \eqref{pf:wx} to each of the particle pairs $(k, k+1), (k+2, k+3), \ldots, (K-1, K)$. To see that the conditions of Lemma \ref{l:u:epsi:dipole} are met, we observe from \eqref{pf:yk} that the lower bound on $z$ in Lemma \ref{l:u:epsi:dipole} is met, and from \eqref{pf:ym} that for each particle pair, the left particle has negative orientation and the right particle has positive orientation.

\paragraph{Case 2: $\# I$ is odd and $b_k = 1$.} This case follows from a minor modification to Case 1. Again, we skip the construction of $\hat v^\ell_\e$, and construct an upper bound of the form \eqref{pf:xs}.

The first modification is that for an odd number of particles the choice for the index $i \in I$ of the surviving particle $x_i$ is not unique. For the construction below it is convenient to take $i = k$, i.e. 
\[
  S_1 = \{1, \ldots, k\} \cup \{K+1, \ldots, N\}.
\]
Then, we apply Lemma \ref{l:u:epsi:dipole} to \eqref{pf:wx} to bound all contributions from the pairs $(k+1, k+2), (k+3, k+4), \ldots, (K-1, K)$. This yields an upper bound of the form \eqref{pf:xs}. 

\paragraph{Case 3: $\# I$ is odd and $b_k = -1$.} This case follows by a similar argument as in Case 2. The only difference is that here we choose $x_K$ instead $x_k$ as the surviving particle.

\paragraph{Case 4: $\# I$ is even and $b_k = 1$.} Since this case covers the simple collision (i.e.\ when $\# I = 2$) for which we introduced a rather elaborate construction of $\hat v_\e^3$ in Subsection \ref{subsubvtildel}, we cannot avoid the construction of $\hat v_\e^3$. However, we can simplify the setting to the two-particle collision in Subsection \ref{subsubvtildel} by removing all but 2 particles from $I$. We do this by proving \eqref{pf:zx:IC} in two steps, i.e.\ we are going to construct $\hat v_\e^0 : \R \to \R$, and from there $\hat v_\e^\ell$, such that  
\[
  \o v_\e (\o T_1, \cdot)
  \leq \hat v_\e^0
  \leq \hat v_\e^\ell (\o T_1, \cdot).
\]   

To establish the first inequality, as in Cases 1, 2 and 3, we apply Lemma \ref{l:u:epsi:dipole} to the pairs $(k+1, k+2), (k+3, k+4), \ldots, (K-2, K-1)$. This yields
\[
  \o v_\e (\o T_1, x)
  \leq \sum_{i \in \hat S_1} ( u - \ep\cs_i^\e \psi ) \left(\displaystyle \frac{x - \xs_i^\e}{\ep} ;  b_i\right) + C \frac{\e^2}{\theta_\e^2} =: \hat v_\e^0(x),
\]
where 
\[
  \hat S_1 :=  \{1, \ldots, k\} \cup \{K, \ldots, N\}.
\]
In this way, we reduce to the case where the only two particles colliding at $T_1$ are $x_k$ and $x_K$. 
From $\{ \xs_i^\e \}_{i \in \hat S_1}$ and $\hat v_\e^0$ we can proceed analogously as in Subsection \ref{subsubvtildel} with the construction of $\hat v_\e^\ell$ on $[\o T_1, \hat T_1] \times \R$ (see Remark \ref{r:Thetae} for a minor modification to the proof).
\smallskip

Since in all four cases the construction of $\hat v_\e^\ell$ can be reduced to that in Section \ref{s:pf:15}, we can construct $\tv_\e$ and $\o v_\e$ beyond $\hat T_1$ analogously to Subsection \ref{s:pf:1:3}. Consequently, \eqref{pf:zz} holds on $(T_1, T_2)$.

\subsubsection{Proof of \eqref{pf:zz} at $t = T_1$.} 
\label{proofatT1secgeneral}
Inequality \eqref{pf:zz} at $t = T_1$ follows from obvious, minor modifications to the arguments in Sections \ref{s:pf:1:4} and \ref{-+seclimsupproof}; we omit the details.

\subsection{The general case}
\label{s:pf:3} 

In this section we prove \eqref{pf:zz} at $T_1$ for general $\bx^0 \in \Omega^N$. One example of a scenario which is not covered in the preceding sections is illustrated in Figure \ref{fig:trajs} at $t = T_1$. Let $M$ be the number of collisions at $T_1$ (note that $M \leq N/2$ ); we label them $m = 1,\ldots,M$. Let 
\[
  y_1 < y_2 < \ldots < y_M
\]
be the spatial points where a collision takes place. Similar to Section \ref{s:pf:2}, set for each $m$
\[
  I_m := \{ k_m, k_m+1, \ldots, K_m \}
\]
as the set of indices of the particles that collide at time-space point $(T_1, y_m)$. Note that the sets $\{I_m\}_m$ are disjoint, i.e.\ $K_m < k_{m+1}$ for all $m$.

We start from $\tv_\e$ and $\o v_\e$ as constructed in Section \ref{s:pf:1}. Then, we construct $\hat v_\e^\ell$ in Subsection \ref{s:pf:3:1} by modifying the construction in Subsections \ref{subsubvtildel} and \ref{s:pf:2:1}. We prove \eqref{pf:zz} at $T_1$ in Subsection \ref{s:pf:3:2}.

\subsubsection{Construction of $\hat v_\e^\ell$, $\tv_\e$ and $\o v_\e$ on $[\o T_1, \o T_2]$} 
\label{s:pf:3:1}

As in Subsection \ref{s:pf:2:1}, we start from $\o v_\e$ as given in \eqref{pf:yi}. Here, \eqref{pf:yn}--\eqref{pf:yk} turn into
\begin{equation*} 
  \theta_i(T_1) \geq l_0
  \quad \text{for all } i \text{ such that } x_i(T_1) < x_{i+1}(T_1)
\end{equation*}
for some fixed $l_0 > 0$, and for $\e$ small enough
\beq\label{pf:yf}  
\xs_{i+1}^\ep- \xs_{i}^\ep
\geq \frac{l_0}2 
\quad \text{for all } i \text{ such that } x_i(T_1) < x_{i+1}(T_1),
\eeq
\beqs 
\theta_\e
=
\o \theta (\o T_1 - \o \tau_\e)
\leq\max_{1 \leq m \leq M} \xs_{K_m}^\ep - \xs_{k_m}^\ep
=: \Theta_\ep = o_\e(1) 
\eeqs
and
 \beq \label{pf:ww}   
   |\cs_i^\ep| 
   \leq C + \begin{cases}
   \dfrac1{\theta_\ep}
   & \text{if } k_m \leq i \leq K_m \text{ for some } m \\
   0 & \text{otherwise.}
\end{cases} 
\eeq   

We recall from Subsection \ref{s:pf:2:1} that it depends on the type of collision (we separated 4 types; see Figure \ref{fig:mult-part-coll}) whether $\hat v_\e^\ell$ needs to be constructed or not. Based on this, we split three cases depending on the number of collisions which require the construction of $\hat v_\e^\ell$, i.e.\ the number of collisions which are as in Case 4 in Subsection \ref{s:pf:2:1}:

\paragraph{Case A: no collision $m$ is as in Case 4.} Thanks to the separation condition \eqref{pf:yf}, we can apply the argument for the bound in \eqref{pf:xs} simultaneously for each collision as in either Case 1, 2 or 3 in Subsection \ref{s:pf:2:1}. This yields
\[
v_\ep(\o T_1,x)\leq  \o v_\ep(\o T_1,x) 
\leq \sum_{i \in S_1} u\left(\displaystyle \frac{x - \o x_i^\e}{\ep} ; b_i\right) 
   + \tilde \rho_\e
   \qquad \text{for all } x \in \R,
\]
where $\tilde \rho_\e = o_\e(1)$ is the sum of all the errors made at each collision $m$, and $S_1$ is a possible choice of surviving particles after the first collision consistent with $\bx$ (recall the nonuniqueness for the solution $\bx$ induced by collisions between an odd number of particles).
In particular, no $\hat v_\e^\ell$ needs to be constructed.

\paragraph{Case B: precisely one collision $m_4$ is as in Case 4.} As in Case 4 from Subsection \ref{s:pf:2:1}, we first construct a profile $\hat v_\e^0$ such that  
\begin{equation} \label{pf:wv} 
  \o v_\e (\o T_1, \cdot)
  \leq \hat v_\e^0.
\end{equation}
This construction goes as in Case A for each $m \neq m_4$ and as in Case 4 for $m = m_4$. Then, the resulting profile $\hat v_\e^0$ has the same properties as in Case 4, and thus we construct $\hat v_\e^\ell$ analogously as done in Case 4.

\paragraph{Case C: there are at least two collisions as in Case 4.} Proceeding as in Case B, we obtain \eqref{pf:wv} for some profile $\hat v_\e^0$ related to the set of particles $\{ \o x_i^\e \}_{i \in \hat S_1}$ for some $\hat S_1$. By this construction, the particles $\{ \o x_i^\e \}_{i \in \hat S_1}$ are for $\e$ small enough separated by $l_0/2$, except for those pairs that corresponds to a collision as in Case 4. Hence, by starting from $\hat v_\e^0$ with the construction of $\hat v_\e^\ell$, we may reduce to the case in which all collisions are simple and of type `$+-$' (see Section \ref{s:pf:15} for this terminology). In more detail, by repeating the constructions of  Cases 1-4 above, we reduce to having the set of colliding particles as follows,
$$ I_m=\begin{cases} \emptyset&\text{if } \# I_m\text{ is even and }b_{k_m}=-1\text{ (as in Case 1)}\\
\{k_m\}&\text{if } \# I_m\text{ is odd  and }b_{k_m}=1\text{ (as in Case 2)}\\
\{K_m\}&\text{if } \# I_m\text{ is odd  and }b_{k_m}=-1 \text{ (as in Case 3)}\\
\{k_m,K_m\}&\text{if } \# I_m\text{ is even and }b_{k_m}=1 \text{ (as in Case 4)}\\
\end{cases}
$$
Therefore, with $2\leq\tilde M\leq M$ as the number of Case-4 collisions, we may assume that there are only $\tilde M$ simple (only two particles collide) collisions 
as in Case 4 and that the number of particles corresponding to the profile $\hat v_0^\e$ is equal to $\tilde N := \# S_1 + 2\tilde M$. 

By relabelling the indices and resetting $n = \tilde N$ and $M = \tilde M$, we may assume that $\hat v_\e^0$ is of the form \eqref{pf:yi} with respect to some $\bx_i^\e$ which satisfies \eqref{pf:yf}--\eqref{pf:ww} with $K_m = k_m + 1$ and $b_{k_m} = 1$ for all $m \in \{1, \ldots, M\}$. 
Then  we set 
\[
  \Theta_\e^m := x_{k_m + 1} - x_{k_m} \leq \Theta_\e
  \quad \text{for all } m.
\]
Instead of the ODE system \eqref{pf:zs} for $\hat \bx$, we consider the same ODE with a slightly different initial condition:
 \beq\label{pf:yb} 
\left\{ \begin{aligned}
   \dot{\xss}_i 
   &= \displaystyle\sum_{j \neq i} 
\displaystyle\frac{b_i b_j}{ \xss_i-\xss_j} - b_i \hat \sigma
 &&\text{in }(0, \hsT_1) \text{ for } i = 1,\ldots, N \\
 \xss_i(0) 
 &= \xs_{i}^\ep - b_i \Theta_\ep 
 &&\text{if } i \neq k_m+1 \text{ for each } m \\
 \xss_i(0) 
 &= \xs_i^\ep - b_i (L \Theta_\e + (\Theta_\e - \Theta_\e^m))
 &&\text{if there exists an } m \text{ such that } i = k_m+1.
 \end{aligned} \right. 
\eeq
The addition of the small value $\Theta_\e - \Theta_\e^m > 0$ to the initial condition in \eqref{pf:yb} is not essential; we do it for convenience such that all particles which are close to collision are separated by the same distance, i.e.
\begin{equation} \label{pf:wu} 
  \xss_{k_m + 1}(0) - \xss_{k_m}(0) = (L+2) \Theta_\e
  \quad \text{for all } m. 
\end{equation}

Lemma \ref{l:te} and its proof apply with minor modifications. Indeed, the proof applies to each of the $M$ collisions separately; the arguments are based on the fact that for any collision $m$, the other $n-2$ particles remain an $\e$-independent, positive distance away from the colliding particles $\hat x_{k_m}$ and $\hat x_{k_m + 1}$ during $[0, \hsT_1)$. Moreover, thanks to \eqref{pf:wu}, the estimates on $\hsT_1$ do not depend on $m$, and thus the value of $\hat \tau_\e$ (and then also $\hat T_1 = \o T_1 + \hat \tau_\ep$) is independent of $m$. 

With the equivalent of Lemma \ref{l:te} established, the remainder of the argument in Subsection \ref{subsubvtildel} for the construction of $\hat v_\e^\ell$ holds with minor modifications. The main modification among them applies to \eqref{pf:ze}, which does not hold for all $x \leq (\o x_{k_m}^\e + \o x_{k_m + 1}^\e)/2$ and for all $m$. However, it does hold whenever $x \leq (\o x_{k_1}^\e + \o x_{k_1 + 1}^\e)/2$ or whenever $( \o x_{k_m} +  \o x_{k_m + 1})/2 \leq x \leq (\o x_{k_{m+1}}^\e + \o x_{k_{m+1} + 1}^\e)/2$ for some $m \in \{1,\ldots,M-1\}$. The precise choice for the lower bound on $x$ is not important; we choose the midpoint between two consecutive collision points because it is independent of $\e$ and because it also appears in the construction of $\hat v_\e^2$.
From the construction of $\hat v_\ep^\ell$, as in Subsection \ref{subsubvtildel},  we infer that 
\begin{equation*}
  v(\hat T_1,x)
   \leq \sum_{i \in S_1} u\left(\displaystyle \frac{x - \hat x_i^\e}{\ep} ; b_i\right) + \hat \rho_\e
   \qquad \text{for all } x \in \R,
\end{equation*}  
with $ \hat \rho_\e=o_\ep(1)$ and $\hat x_i^\e =o_\ep(1)$. 
\smallskip

From the results in all three cases, the construction in Subsection \ref{s:pf:1:3} applies to build $\tv_\e$ and $\o v_\e$ beyond $\hat T_1$. Consequently, \eqref{pf:zz} holds on $(T_1, T_2)$.

\subsubsection{Proof of \eqref{pf:zz} at $t = T_1$}
\label{s:pf:3:2}
As in Section \ref{proofatT1secgeneral}, we refer for the proof of \eqref{pf:zz} at $t = T_1$ to Sections \ref{s:pf:1:4} and \ref{-+seclimsupproof} and leave the obvious modifications to the reader. 

\subsection{Induction step} 
The induction statement $P_k$ is: 
\begin{center}
$\o v_\e (\o T_k, x)$ is as in \eqref{pf:yi} with some parameters $\o c_i^\e$ and $\o x_i^\e$ for $i \in S_{T_{k-1}}$ which satisfy  \eqref{pf:yf}-\eqref{pf:ww} for all $\e$ small enough, and \eqref{pf:zz} holds on $[T_{k-1}, T_k)$.
\end{center}
In Section \ref{s:pf:3} we prove that $P_1$ implies $P_2$. This proof does not depend on the parameter $k$, and thus we may iterate it finitely many times until $T_{K+1} = T$. This completes the proof of Theorem \ref{t}.

\bigskip 
\noindent {\bf Acknowledgements.}  The first author has been supported by JSPS KAKENHI Grant Number JP20K14358. The second  author has been supported by the
NSF Grant DMS-2155156 ``Nonlinear PDE methods in the study of interphases".

\end{document}